\documentclass[11pt]{amsart} 
\usepackage[margin=1in]{geometry} 
\usepackage{setspace} 
\usepackage{amsmath, amsthm, amssymb, amscd}
\usepackage{cases}
\usepackage{amsfonts}
\usepackage[utf8]{inputenc}
\usepackage{enumitem}
\usepackage{mathrsfs}
\usepackage{mathtools}
\usepackage{exscale}
\usepackage{relsize}
\usepackage{tikz-cd}
\setlist[enumerate]{label=\textnormal{(\roman*)}}
\usepackage{verbatim}
\usepackage{bm}
\usepackage{hyperref}
\usepackage{fancyhdr}
\usepackage{setspace}

\newtheorem{theorem}{Theorem}[section]
\newtheorem{corollary}[theorem]{Corollary}
\newtheorem{lemma}[theorem]{Lemma}
\newtheorem{proposition}[theorem]{Proposition}

\theoremstyle{definition}
\newtheorem{definition}[theorem]{Definition}
\newtheorem{remark}[theorem]{Remark}
\newtheorem{conjecture}[theorem]{Conjecture}
\newtheorem{example}[theorem]{Example}

\numberwithin{equation}{section}
\newcommand\Pl{\mathbb{P}^{1}}

\newcommand*{\dif}{\mathop{}\!\mathrm{d}}

\newcommand{\MIC}{\mathrm{MIC}}
\newcommand{\ie}{\textit{i}.\textit{e}.}

\newcommand{\para}{\textnormal{par}}

\newcommand{\Fil}{\textnormal{Fil}}

\begin{document}	
\title{Constructing Parabolic Non-Abelian Hodge Correspondence in Positive Characteristic Using Parabolic Bases} 
\author{Xiaojin Lin}	
\email{xjlin@mail.ustc.edu.cn}

\begin{abstract}
    We introduce the concept of parabolic bases to establish a localized framework for parabolic bundles and parabolic $\lambda$-connections. 
    Building on this foundation, we propose a novel method for constructing the parabolic non-abelian Hodge correspondence in positive characteristic, extending the work originally developed by Krishnamoorthy and Sheng for algebraic curves. 
    Additionally, we investigate the rank $2$ parabolic Higgs-de Rham flow operator and present a modified version of the Sun-Yang-Zuo algorithm, specifically adapted to the parabolic setting.


\end{abstract}

\maketitle
\tableofcontents
\section{Introduction}

    \subsection{Motivation and Background}
    In their groundbreaking research \cite{MeSe80Moduofvbwithparstr}, Mehta and Seshadri introduced the innovative concept of parabolic bundles.
    They showed that the category of irreducible unitary representations of the fundamental group of a punctured compact Riemann surface $(C,D)$ is equivalent to the category of stable parabolic bundles with parabolic structures concentrated on $D$ and a zero parabolic degree on $C$.
    This notable progress broadens the scope of the Narasimhan-Seshadri correspondence, which was originally formulated for compact Riemann surfaces.
    
    Recent studies on the geometry of logarithmic algebraic varieties have highlighted the crucial significance of parabolic bundles, especially within non-abelian Hodge theory on non-compact varieties. 
    Simpson’s work is of notable significance: Classical non-abelian Hodge theory established connections among semisimple flat bundles, polystable Higgs bundles with vanishing Chern classes, and local systems.
    In \cite{Sim90Harbdonnoncomcur}, Simpson extended the non-abelian Hodge correspondence to non-compact curves by involving filtered Higgs bundles, filtered de Rham bundles, and filtered local systems.
    The filtered Higgs bundle and the filtered flat connection correspond to their parabolic equivalents, known as the parabolic Higgs bundle and the parabolic flat connection, respectively.
    
    It is natural to seek an analogous extension for the non-abelian Hodge correspondence to parabolic setting in characteristic $p$, building upon the remarkable work of Ogus-Vologodsky \cite{OV2007NonAbeHincharp} and Schepler \cite{Sche08LognonabeHodincharp} for the logarithmic case, which constitutes the principal subject of this paper.
    An essential component of the Ogus-Vologodsky theory is the development of the inverse Cartier transform functor along with its quasi-inverse. This paper is chiefly dedicated to the construction of the parabolic Cartier transform functor and its quasi-inverse, as initially proposed by \cite{KrisSheng20perideRhamovercur}.
    
    In fact, the investigation of the parabolic inverse Cartier transformation is driven by the quest to understand the Higgs-de Rham flow, transcending the mere generalization of the original inverse Cartier transformation.
  The concept of Higgs-de Rham flow is fundamentally linked to the earlier introduction of periodic Higgs bundles by Mao Sheng and Kang Zuo, as detailed in \cite{ShengZuo12periodHiggs}. 
    Subsequently, in \cite{LSZ19SsHiggsbdPerbdrepoffdgp}, a systematic development of the Higgs-de Rham flow theory was achieved, based on the foundational work of the aforementioned Ogus-Vologodsky theory. 
    This theory can be seen as a positive characteristic (and $p$-adic) counterpart to the Yang-Mills flow, which is pivotal in various research domains, including the study of the positive Bogomolov inequality \cite{Lan15Bogoineincharp}, $p$-adic uniformization theory \cite{LSYZ19uniofpadiccurveviaflow}, the theory of rigid connections \cite{EsGro20RigidconandFiso}, and other related areas.
    
    The periodic Higgs-de Rham flows are intricately linked to Fontaine-Faltings modules with a certain endomorphism structure. 
    Fontaine-Laffaille modules constitute the $p$-adic analogue to the polarized variations of Hodge structures, thereby serving as a pivotal tool in the study of $p$-adic Hodge theory within the realms of algebraic geometry and number theory. 
    Furthermore, this advancement provides novel insights into the arithmetic exploration of the Gauss-Manin connection, a field initially investigated by Zuo \cite{Zuo18arithSim} and Li and Sheng \cite{LiSheng22CharofBeanumber}, among others.
    The foundational principle underlying the study of the Gauss-Manin system via the Higgs-de Rham flow is based on the result that the Gauss-Manin system inherently induces a periodic (parabolic) Higgs-de Rham flow (see Krishnamoorthy and Sheng \cite{KrisSheng20perideRhamovercur}).
    A series of studies focusing on specific cases has been further investigated, particularly those involving rank $2$ flows on projective curves.
    For example, Li and Sheng's research on genus zero modular curves via the Higgs-de Rham flow \cite{LiSheng22CharofBeanumber}, and Yang and Zuo's study \cite{YangZuo23Constrfamiavof} on motivic rank $2$ local systems on a projective line minus four points based on \cite{SYZ22Projecrysrepoffmtgpandtwistedflow}.
    In these specific instances, the arithmetic of the Gauss-Manin system is derived from its periodic characteristics.
    
    This exploration introduces the concept of the parabolic Higgs-de Rham flow.
    One of the reasons to expand the original theory is as follows: The monodromy of the Gau\ss-Manin system at $\infty$ is typically quasi-unipotent but not unipotent. 
    Therefore, the residues of the connection at infinity may not be nilpotent, which poses an obstacle to constructing the associated periodic Higgs-de Rham flow.
    To tackle this problem, it is crucial to extend the connection to a parabolic one on compactification and to develop flow theory to accommodate parabolic contexts.
    
    In \cite{SYZ22Projecrysrepoffmtgpandtwistedflow} Sun,Yang and Zuo studied the twisted Higgs-de Rham flow, which can be interpreted as a special type of parabolic Higgs-de Rham flow.
    Subsequently, Krishnamoorthy and Sheng \cite{KrisSheng20perideRhamovercur} systematically developed a theory of positive characteristic parabolic Higgs-de Rham flows over curves. 
    The pivotal step in their work involves the construction of the parabolic inverse Cartier transformation. 
    Their construction exploits the so-called \emph{Biswas-Iyer-Simpson} (BIS for brevity) correspondence, which establishes a relationship between parabolic $\lambda$-connections and orbifold $\lambda$-connections.
    The parabolic inverse Cartier on a log variety is delineated as the descent of the inverse Cartier on a specific finite cover via the BIS correspondence in essence. 
    Moreover, it can be shown that the parabolic inverse Cartier, defined in this manner, is intrinsic.
    
    \subsection{Main Results}
This research introduces a straightforward technique for developing the parabolic inverse Cartier transformation applicable to any dimension. 
    The functor introduced in this work will be demonstrated to align with the dimension one scenario previously described in \cite{KrisSheng20perideRhamovercur} (refer to Corollary \ref{cor of two par inverse cartier coincide}).
  Our approach employs the creation of a localized formal bases concept, referred to as the parabolic bases, to describe the parabolic structure of the parabolic bundle.
    This allows us to directly use the exponential twist technique in formulating the inverse Cartier transformation.
    
    Specifically, let $(X,D)$ be a smooth log variety, we present a new category, represented by $\mathcal{PVB}/(X,D)$, which is defined as the category of vector bundles equipped with parabolic bases.
    Within this category, the objects are referred to as vector bundles equipped with parabolic bases (as detailed in Definition \ref{def of global vb with pb}) over $(X,D)$.
    The parabolic bases over an open subscheme (sufficiently small) $U$ of $X$ are formally characterized as bases of the underlying bundle, twisted by a \textit{fractional exponent} of the defining sections of the divisor $U \cap D$. 
    Subsequently, we describe the transition data of the parabolic bases across overlapping subsets, as the open subschemes vary, to construct a global object.

    The initial finding (refer to Section \ref{the section: establising equivalence}) demonstrates that the category $\mathcal{PVB}$ is equivalent to the category of parabolic bundles. 
    The explicit natural isomorphism is constructed within the proof.            
    The advantage of using parabolic bases is its capability to align the definition more closely with the original conception of a vector bundle. 
    Additionally, we will demonstrate that this approach also streamlines the handling of these bundles. 
    For example, in the category $\mathcal{PVB}$, one can define tensor (refer to Definition \ref{def of tensor of par}) and $\mathcal{H}om$ (see Definition \ref{def of hom of par b}) functors in an intuitive manner.
    This is consistent with the original definition of a parabolic bundle under equivalence, as stated in Proposition \ref{prop of equi preserves Hom and tensor}.

    In conclusion, we arrive at the following outcome:
    \begin{theorem}[=Theorem \ref{thm of equivalence between two category}]\label{intro: thm of equivalence between two category}
    The category of vector bundles with parabolic bases over $(X,D)$ is tensor equivalent to the category of parabolic bundles over $(X,D)$.
    \end{theorem}

    We then extend the category $\mathcal{PVB}$ to a new category, denoted as $\lambda$-$\mathcal{PC}on$, which encompasses $\lambda$-connections equipped with parabolic bases.
    We prove that this newly-defined category is equivalent to the category of parabolic $\lambda$-connections. 
    Within $\lambda$-$\mathcal{PC}on$, we introduce the concept of adjusted objects, which correspond to the adjusted parabolic $\lambda$-connection under the equivalence. 
    This concept, initially presented by Iyer and Simpson in \cite{IyerSim07Arelationbetweenparchern}, is crucial in Hodge theory.
    Hence, Theorem \ref{intro: thm of equivalence between two category} is enhanced as follows.
    \begin{theorem}[=Theorem \ref{thm of equiv between adjusted par and nil conn with par bas}]
        The category of parabolic $\lambda$-connections with parabolic bases over $(X,D)$ is tensor equivalent to the category of $\lambda$-connections over $(X,D)$.
      Moreover, the adjusted $\lambda$-connections correspond to the adjusted $\lambda$-connections with parabolic bases under this equivalence.
    \end{theorem}
    
    Furthermore, we present the definitions of pullback and pushforward for $\lambda$-connections with parabolic bases. 
    A pullback is designed for flat morphisms, while a pushforward is associated with finite surjective morphisms, generally referred to as branched coverings.
    Thus, the functoriality property  for the parabolic bases is confirmed.
    The primary outcome of the local theory of parabolic bundles is the following theorem.
    
    \begin{theorem}[=Theorem \ref{thm of Bis for vb with pb}]
        If $f: (X,D) \to (Y,B)$ is a finite branched Galois covering between smooth varieties over $k$, such that the ramification indices are all non-zero in $k$.
        Then we have the following equivalence between two categories:
        \begin{equation*}
        \begin{tikzcd}
        \left\{ \vcenter{\hbox{ $  G$-equivariant parabolic}\hbox{\qquad $\lambda$-connections}}  \right\} \Bigm/(X,D)  & \{\textit{parabolic }\lambda\textit{-connections}\}/(Y,B)
	\arrow["{(f_{\para,*})^{G}}"{pos=0.6}, shift left, harpoon, from=1-1, to=1-2]
	\arrow["{f^{*}_{\mathrm{par}}}", shift left, harpoon, from=1-2, to=1-1]
        \end{tikzcd}.
        \end{equation*}
        Moreover, the $G$-equivariant adjusted (strongly) parabolic $\lambda$-connections on $(X,D)$ correspond to the adjusted (strongly) parabolic $\lambda$-connections on $(Y,B)$ under the equivalence.
    \end{theorem} 
    The theorem expands the \textit{Biswas-Iyer-Simpson correspondence} (abbreviated as BIS correspondence). 
    Particularly, when limited to the subcategory of $G$-equivariant $\lambda$-connections that have a trivial parabolic structure on the left, it coincides with the BIS correspondence.
    
    Next, we use the developed local theory of parabolic bundles to investigate the parabolic non-abelian Hodge correspondence in positive characteristic. 
    Let $\mathrm{HIG}_{lf,*}^{p-1}$ denote the category of parabolic Higgs bundles over the pair $(X,D)$, where the Higgs fields are nilpotent with a nilpotence exponent of at most $(p-1)$, and their residues are also nilpotent with a nilpotence degree not exceeding $(p-1)$, $\mathrm{MIC}_{lf,*}^{p-1}$ represents the category of flat bundles on $(X,D)$ where both their $p$-curvatures and the residues along each component of $D$ are nilpotent with a nilpotence exponent less than $(p-1)$.
    The positive characteristic parabolic non-abelian Hodge correspondence establishes an equivalence between $\mathrm{HIG}_{lf,*}^{p-1}$ and $\mathrm{MIC}_{lf,*}^{p-1}$, achieved through the construction of parabolic Cartier and parabolic inverse Cartier transformation functors.     
    The primary technique we employ is the exponential twist, as originally introduced in \cite{LSZ15NonabeHodexptwi}.
    
    In regard to the inverse Cartier transformation, following the approach of Lan, Sheng and Zuo \cite{LSZ15NonabeHodexptwi}, we apply the exponential twist to the Frobenius twist of the parabolic Higgs bundle, making use of the parabolic bases in addition to its associated Higgs field and transition data. 
    We show that the output data form a parabolic connection.
    The next steps maintain alignment with the initial case, which culminates in achieving a flat parabolic connection, thereby completing the inverse Cartier construction.

    To build the parabolic Cartier transformation, we must first define the parabolic Cartier descent functor, which serves as a quasi-inverse to the Frobenius pullback. 
    This result carries distinct importance, and we have adapted Katz's seminal proof \cite{Katz70Nilpotneconnections} to the parabolic bases in our methodology.
    Hence we get the Cartier descent theorem for the parabolic setting.
    \begin{theorem}[=Theorem \ref{thm of Cartier Descent}]
        The following two categories are equivalent 
        \begin{equation*}
        \begin{tikzcd}
        \{  \text{Parabolic vector bundles } V'_{*} \} /(X',D')  & \left\{ \vcenter{\hbox{Strong parabolic $\lambda$-connections $(V_{*},\nabla)$} \hbox{\qquad with vanishing $p$-curvature}}  \right\} \Bigm/(X,D)
	\arrow["{F^{*}}"{pos=0.6}, shift left, harpoon, from=1-1, to=1-2]
	\arrow["{( \, )^{\nabla}}", shift left, harpoon, from=1-2, to=1-1]
        \end{tikzcd}.
        \end{equation*}
        where $F^{*}$ denotes the Frobenius pullback, and $( \, )^{\nabla}$ is a functor that takes horizontal sections.
    \end{theorem}
    In alignment with \cite{LSZ15NonabeHodexptwi}, for a flat connection $(H,\nabla)$, we employ an exponential twist specific to the parabolic bases, producing a triple $(H',\nabla',\psi(\nabla))$. 
    Here, $(H',\nabla')$ represents a parabolic connection with vanishing $p$-curvature, while $\psi(\nabla)$ is the $p$-curvature of $\nabla$, now corresponding with $\nabla'$. 
    Through Cartier descent, the triplet $(H',\nabla',\psi(\nabla))$ descends into a parabolic Higgs bundle, thereby concluding the parabolic Cartier transformation.
    The main result is as follows:

    \begin{theorem}[=Theorem \ref{thm of catier correspondence}]
    We have the following equivalence of categories,
        \begin{equation*}
        \begin{tikzcd}
            {\{\mathrm{MIC}_{lf,*}^{p-1}\}/(X,D)} & {\{\mathrm{HIG}^{p-1}_{lf,*}\}/(X',D')}
            \arrow["{C_{\exp}}", shift left, harpoon, from=1-1, to=1-2]
            \arrow["{C_{\exp}^{-1}}", shift left, harpoon, from=1-2, to=1-1]
        \end{tikzcd}.
        \end{equation*}
    \end{theorem}
                
    Inspired by the research conducted by Alfaya and Biswas \cite{AlBis23pullandpshofparbd}, we investigate the functorial properties of the parabolic inverse Cartier transformation. 
    Specifically, we demonstrate that the pushforward and pullback operations are commutative with this transformation.
    \begin{theorem}[=Theorem \ref{thm of par inverse commutes with pullback and push forward}]
        Let $X$ and $Y$ be two smooth curves over a field $k$ of characteristic $p$, and $f$ be a separable finite morphism from $X$ to $Y$ with a $W_{2}(k)$-lifting $\tilde{f}: \tilde{X} \to \tilde{Y}$, and we assume that all ramification indices are co-prime to $p$.
        Let $D$ (resp. $\tilde{D}$) be a divisor of $Y$ (resp. $\tilde{Y}$), and let $B=(f^{*}D)_{\mathrm{red}}$ (resp. $\tilde{B}$) be the reduced divisor of $f^{*}D$ (resp. $\tilde{f}^{*}\tilde{D}$).
        Then the functors $f_{\para,*}$ and $f_{\para}^{*}$ commute with parabolic $C^{-1}$ and $C$ transformations.
        Moreover, if $f$ is Galois, $C^{-1}$ and $C$ also commute with $(f_{*})^{G}$.
    \end{theorem}
           
The remainder of this paper focuses on the parabolic Higgs–de Rham flow operator of rank 2, acting on Higgs bundles over a logarithmic curve \((X_1, D_1)\). 
    We restrict our discussion to Higgs bundles of rank 2; while this represents the simplest non-trivial case, it is deeply intertwined with the theory of \textit{opers} and carries significant applications across diverse areas of mathematics. 
    Notable examples include the work of Lan, Sheng, Yang, and Zuo on the uniformization theory of $p$-adic curves \cite{LSYZ19uniofpadiccurveviaflow}, as well as the contributions of Yang and Zuo to the study of motivic rank $2$ local systems on the projective line minus four points \cite{YangZuo23Constrfamiavof}, building on the foundational work in \cite{SYZ22Projecrysrepoffmtgpandtwistedflow}.            
    
    The flow operator is explicitly defined as
    \begin{equation*}
    \mathrm{Gr}_{\Fil} \circ C^{-1}(E,\theta), 
    \end{equation*}
    where $(E,\theta)$ is a semistable graded Higgs bundle (See \eqref{def of gaded Higgs bd} for the precise definition) and $\Fil$ is a Simpson filtration on $(H,\nabla):=C^{-1}(E,\theta)$.            
    In the rank $2$ case, the Simpson filtration of $(H,\nabla)$ coincides with the parabolic Harder-Narasimhan filtration of $H$.
    The main result of this work is to provide an algorithm (Algorithm \ref{algorithm}) for determining the maximal destabilizing subbundle of $H$, which enables the explicit computation of the flow operator.
    The algorithm adapted from Sun, Yang and Zuo \cite[Appendix A]{SYZ22Projecrysrepoffmtgpandtwistedflow}, with its generalization based on the following observation:
    The exactness of the inverse Cartier transformation induces the following short exact sequence:
    \begin{equation*}
        0 \to F^{*}L_{2} \to H \to F^{*}L_{1} \to 0.
    \end{equation*}
    Let $\xi \in \mathrm{Ext}^{1}(F^{*}L_{1},F^{*}L_{2})$ denote the corresponding extension class. 
    The following relationship holds:
    \begin{theorem}[=Theorem \ref{thm of ext class}]
        The extension class $\xi$ satisfies the equation
        \begin{equation}
            \xi = F^{*}\theta \cup \kappa,
        \end{equation}
        where $\kappa$ is the Deligne-Illusie class associated with the $W_{2}(k)$ lifting $(X_{1},D_{1}) \hookrightarrow (X_{2},D_{2})$.
    \end{theorem}
    The theorem provides a clear characterization of how the bundle transforms under the inverse Cartier transformation. 
    Moreover, as a direct consequence, we can identify the maximal destabilizing subbundle of $H$, as demonstrated in Proposition \ref{m d s}. 
    In the appendix, we simplify the Li-Sheng conjecture using our algorithm, which also serves as a specific instance of the broader algorithmic framework.

    \subsection{Organization}
The organization of the paper is as follows.

    In Section \ref{Local Theory for Parabolic Bundle}, we introduce the category of vector bundles with parabolic bases and show that the category is equivalent to the category of parabolic bundles.
    
    In Section \ref{Local theory of par lam con}, we extend this equivalence to the category of parabolic $\lambda$-connections. 
    We also define pushforward and pullback functors for these objects, thereby establishing further local properties.

    In Section \ref{Par Higgs-de Rham flow}, we develop the parabolic non-abelian Hodge correspondence in positive characteristic by employing the concepts introduced earlier. 
    Additionally, this section provides an algorithm specifically for the rank $2$ Higgs-de Rham operator.

    Finally, in Section \ref{Discussion}, we discuss potential applications of our results and outline promising directions for future research.
```

\section{Local Theory for Parabolic Bundles}\label{Local Theory for Parabolic Bundle}

\subsection{Vector Bundles with Parabolic bases}
This section presents the concept of parabolic bases.
Suppose $X$ is a smooth variety over a field $k$ and $D$ is a normal crossing divisor.
Write $D=\sum_{i=1}^{h}D^{i}$ as a union of irreducible components, and we assume that $D_{i}$ are themselves smooth meeting transversally.

```
\textbf{For clarity, we first assume $h=1$.}

We begin with the following short exact sequence 
\begin{equation}\label{ses of O(D)}
0 \to L \stackrel{s_{D}}{\to} L(D) \to \mathcal{O}_{D} \to 0,
\end{equation}
    where $L$ is a line bundle on $X$.
We give a local description of the sequence \eqref{ses of O(D)}.    	
Consider an open cover defined by 
    \begin{equation}\label{open cover}
    \mathcal{U} = \{ U_i \mid i \in I \}
    \end{equation}
    where, for each $U_{i}$, the line bundles $L|_{U_{i}}$ and $L(D)|_{U_{i}}$ are trivial.
We choose the bases $e_{i}$ for $L$ and $f_{i}$ for $L(D)$.
If $U_{i} \cap D = \emptyset$, the morphism $s_{D}|_{U_{i}}:L|_{U_{i}} \to L(D)|_{U_{i}}$ is an isomorphism. 
    Hence, we make an identification of the bases $e_{i}$ and $f_{i}$ on $U_{i}$.

If $U_{i}\cap D \neq \emptyset$, suppose $s_{D}(e_{i})=s_{i}f_{i}$. 
Therefore, according to the short exact sequence \eqref{ses of O(D)}, the locus where $s_{i}$ vanishes is exactly $U_{i} \cap D$.
Thus, we can define $f_{i}$ as $e_{i}/s_{i}$ (where $s_{D}$ is omitted for simplicity). 
    We then proceed to compare the transition data using this identification.            
    Suppose the transition information can be expressed as:
    \begin{equation*}
        \phi(e_{i}) = g_{ij} e_{j}, \quad \psi(f_{i}) = p_{ij} f_{j},
    \end{equation*}
    then the relationship between $g_{ij}$ and $p_{ij}$ naturally emerges from converting the original bases to a parabolic bases, as demonstrated in Table \ref{tab: transition function of L(D)}.

\begin{table}[ht]
    \centering
    \begin{tabular}{|c|c|}
        \hline 
         & Transition data $p_{ij}$ of $L(D)$ \\
        \hline
        $U_{i} \cap D = U_{j} \cap D = \emptyset$ & $g_{ij}$ \\
        \hline
        $U_{i} \cap D = \emptyset, \, U_{j} \cap D \neq \emptyset$ & $g_{ij} s_{j}$\\
        \hline
        $U_{i} \cap D \neq \emptyset, \, U_{j} \cap D = \emptyset$ & $g_{ij} / s_{i}$\\
        \hline
        $U_{i} \cap U_{j} \cap D \neq \emptyset$ & $s_{j}g_{ij}/s_{i}$\\
        \hline
    \end{tabular}
    \caption{The transition data of $L(D)$.}
    \label{tab: transition function of L(D)}
\end{table}

In summary, the bases of the bundle $ L(D) $ can be identified with that of $L$ in the neighborhood outside of $D$. 
    Within the open neighborhood intersecting $D$, the basis of $ L(D) $ corresponds to the bases of $L$ twisted by a section whose vanishing locus is $D \cap X_i$.
The transition data are naturally derived from the base change.      
    The parabolic basis is constructed similarly to the previously outlined bases for $L(D)$.
    However, the twisted component can take the form of a fractional power of the section $s_{i}$.       
    We start by giving the precise definition in the case where the rank is $1$.
Let the notations be as above and let $\alpha$ be a rational number \footnote{We can also make corresponding definitions for $\alpha$ being irrational, but to avoid overly divergent content, we still limit it to the case of rational numbers.} in $\left[0,1\right)$. 
    If $\mathrm{char}(k)=p$, we require the denominator of $\alpha$ to be coprime to $p$.

    Given $L$ and $\mathcal{U} = \{ U_i \mid i \in I \}$ in the previous context, the parabolic basis of $L$ for a given weight $\alpha$, with respect to $\mathcal{U}$, is defined as follows: 
    For each index $i$, if $U_i \cap D$ is not empty, the parabolic basis on $U_i$ is $f_i = e_i / s_i^\alpha$. 
    Conversely, if $U_i \cap D$ is empty, we set $f_i = e_i$.    	
Similarly, formal transition data for the parabolic basis are inherently induced by the base change transformation of transition data associated with $L$. 
    We arrange these transition data in the following Table \ref{tab: transition function of (L,alpha)}.
    \begin{table}[ht]
        \centering
        \begin{tabular}{|c|c|}
            \hline 
             & Transition data $p_{ij}$ of $(L,\alpha)$ \\
            \hline
            $U_{i} \cap D = U_{j} \cap D = \emptyset$ & $g_{ij}$ \\
            \hline
            $U_{i} \cap D = \emptyset, \, U_{j} \cap D \neq \emptyset$ & $g_{ij} s_{j}^{\alpha}$\\
            \hline
            $U_{i} \cap D \neq \emptyset, \, U_{j} \cap D = \emptyset$ & $g_{ij}/s_{i}^{\alpha}$\\
            \hline
            $U_{i} \cap U_{j} \cap D \neq \emptyset$ & $g_{ij}(s_{j}/s_{i})^{\alpha}$\\
            \hline
        \end{tabular}
        \caption{The transition data of $(L,\alpha)$.}
        \label{tab: transition function of (L,alpha)}
    \end{table}  
    
    We denote the line bundle with parabolic bases mentioned above by $(L, \alpha, D)$. 
    When clarity permits, we may simplify this notation to $(L, \alpha)$.  	
Next, we define the morphisms between line bundles equipped with parabolic bases.
    Let $(E,\alpha)$ and $(F,\beta)$ be two line bundles with parabolic bases.
    \begin{definition}\label{def for morphism for vbp of rank 1}
        A morphism $f$ from $(E,\alpha)$ to $(F,\beta)$ is a sheaf morphism $f$ from $E$ to $F$, such that if $\alpha > \beta$, then $D \subset (f)_{0}$.
    \end{definition}
    
    To illustrate the reasonableness of the definition, we translate the aforementioned condition into local data for the reader.
    Let $\mathcal{U} = \{ U_i \mid i \in I \}$ be an open cover of $X$ such that both $E$ and $F$ are free on each $U_{i}$. 
    For a morphism $f \in \mathrm{Hom}(E,F)$, we can locally express $f_{i}: E|_{U_{i}} \to F|_{U_{i}}$ in terms of the basis as $f_{i}(e_{i}) = u_{i}f_{i}$, where $u_{i}$ is a section over $U_{i}$.

    The morphism induced on the parabolic bases is given by:
    \begin{equation}\label{transition data for parabolic bases}
    f_i \left( e_i / s_i^{\alpha} \right) = u_{i} s_i^{\beta-\alpha} \left( e_i / s_i^{\beta} \right).
    \end{equation}
    The condition that $D \subset (f)_{0}$ when $\alpha > \beta$ is equivalent to the requirement that the valuation of $u_{ij} s_{i}^{\beta-\alpha}$ with respect to $D$, valued in $\mathbb{Q}$, is positive. 
    This is a natural and necessary condition.

Next, we define the local parabolic vector bundle with parabolic bases in the higher-rank case.

\begin{definition}[Local model]\label{def of local pb}
        Let \( U \) be an affine scheme, and let \( D \) be an irreducible and reduced divisor on \( U \). 
        Let \( \bm{\alpha} = \{\alpha_1, \ldots, \alpha_r\} \) denote a sequence of rational numbers in the interval \([0, 1)\), with the additional condition that the denominator of each \( \alpha_i \) is invertible in \( \Gamma(U, \mathcal{O}_U) \).
        
        Let \( M \) be a free \( \mathcal{O}_U \)-module of rank \( r \). 
        The \textit{parabolic data} of type \( \bm{\alpha} \) on \( M \) is defined as the tuple:
        \[
        \left( \{e_1, \ldots, e_r\}; \{\alpha_1, \ldots, \alpha_r\} \right),
        \]
        where \( \{e_1, \ldots, e_r\} \) forms a basis for \( M \). 
        
        The associated \textit{parabolic basis} is given by the set:
        \[
        \left\{ \frac{e_1}{s^{\alpha_1}}, \ldots, \frac{e_r}{s^{\alpha_r}} \right\},
        \]
        where \( s \) is a section over \( U \) such that \( (s)_0 = D \). 
        
        For notational convenience, we denote the parabolic data by \( (M, \bm{e}, \bm{\alpha}) \), or more compactly as \( (\bm{e}, \bm{\alpha}) \), where \( \bm{e} = \{e_1, \ldots, e_r\} \) and \( \bm{\alpha} = \{\alpha_1, \ldots, \alpha_r\} \).
    \end{definition}
    
    It is obvious that $(M,\bm{e},\bm{\alpha})$ is a direct sum of line bundles with parabolic bases, \ie, we can write it as $(M,\bm{e},\bm{\alpha})=\oplus_{i=1}^{r} (M_{i},e_{i},\alpha_{i})$, where each $M_{i} \}$ represents the free module generated by $ e_i $.    
    Let $(L,\bm{e},\bm{\alpha})$, $( M,\bm{f},\bm{\beta})$ be two free modules with parabolic bases.
    \begin{definition}\label{def of local morphism for vb with pb}
    A morphism $g$ from $(L,\bm{e},\bm{\alpha})$ to $( M,\bm{f},\bm{\beta})$ is a sheaf morphism from $L$ to $M$ that respects the parabolic data.
    To be more specific, let us assume that $(L, \bm{e}, \bm{\alpha})$ is expressed as a direct sum:
    \begin{equation*}
    (L, \bm{e}, \bm{\alpha}) = \bigoplus_{i=1}^{r} (L_i, e_i, \alpha_i),
    \end{equation*}
    and similarly, $(M, \bm{f}, \bm{\beta})$ is represented as:
    \begin{equation*}
    (M, \bm{f}, \bm{\beta}) = \bigoplus_{j=1}^{s} (M_j, f_j, \beta_j).
    \end{equation*}
    In this setting, each mapping $g_{ij}$ from $(L_i, e_i, \alpha_i)$ to $(M_j, f_j, \beta_j)$ is required to be a morphism between the line bundles with their parabolic bases.
    \end{definition}

    We give an example of Definition \ref{def of local pb}.
    \begin{example}\label{example: toy example of pushforward}
        Consider $X = \mathrm{Spec} A\left[x\right]$ and $Y = \mathrm{Spec} A\left[y\right]$, with a morphism $f$ from $X$ to $Y$ constructed by mapping $y$ to $x^{N}$, where $N$ is a positive integer invertible in $A$.
        Subsequently, we demonstrate the existence of a natural parabolic bases for the bundle $f_{*}\mathcal{O}_{X}$.
        
        Observe that $A\left[x\right]$ forms a free $A\left[y\right]$-module with rank $N$ via the ring homomorphism $f^{\#}$. 
        This enables a canonical decomposition in accordance with the group $\mathrm{Aut}(f)(\cong \mathbb{Z}/N)$ acting as an $A[y]$-module:
        \begin{equation*}
            A\left[x\right] \cong A\left[y\right] \oplus xA\left[y\right] \oplus \cdots \oplus x^{N-1}A\left[y\right],
        \end{equation*}
        where $x^{i}A\left[y\right]$ corresponds to the free $A[y]$-module generated by $x^{i}$.
        For each $x^{i}$, a parabolic weight $i/N$ is assigned. 
        As a result, the corresponding parabolic bases can be expressed as 
        \begin{equation*} 
        \left\{ 1, \frac{x}{y^{1/N}}, \ldots, \frac{x^{N-1}}{y^{(N-1)/N}} \right\}. 
        \end{equation*}
        
        In geometrical terminology, we have accomplished the following: The vector bundle $f_{*}\mathcal{O}_{X}$ is equipped with a tautological parabolic bases characterized by the parabolic data given as:
        \begin{equation*}
        (\{1,x,\cdots,x^{N-1}\};\{0,1/N,\cdots,(N-1)/N\}).
        \end{equation*}                
    \end{example}            
    We shall generalize this toy example to the pushforward of a finite morphism of a locally abelian parabolic $\lambda$ -connection in Section \ref{def of pushforward of par}.

    \begin{definition}[Global definition]\label{def of global vb with pb}
        Let $\bm{\alpha} = \{\alpha_{1}, \ldots, \alpha_{r}\}$, and let $D$ be as defined in Definition \ref{def of local pb}. 
        Let $V$ be a vector bundle over $X$, and let $\mathcal{U} = \{ U_{i} \}_{i \in I}$ be an open affine cover of $X$ such that $V|_{U_{i}}$ is free. 
        The \textbf{parabolic data} of type $\bm{\alpha}$ associated with $V$ is defined as follows:
        
        \begin{itemize}
            \item If $U_{i} \cap D \neq \emptyset$, the parabolic data on $U_{i}$ is given by the tuple
            \[
            (\bm{e}_{i}; \bm{\alpha})_{i} := \left( \{e_{i}^{1}, \ldots, e_{i}^{r} \}; \{\alpha_{1}, \ldots, \alpha_{r}\} \right)_{i},
            \]
            where $\{e_{i}^{1}, \ldots, e_{i}^{r} \}$ is a basis for $V|_{U_{i}}$ over $U_{i}$. 
            The corresponding \textbf{parabolic basis} is the set 
            \[
            \left\{ \frac{e_{i}^{1}}{s_{i}^{\alpha_{1}}}, \ldots, \frac{e_{i}^{r}}{s_{i}^{\alpha_{r}}} \right\},
            \]
            where $s_{i}$ is a section over $U_{i}$ such that $(s_{i})_0 = D \cap U_{i}$.
        
            \item If $U_{i} \cap D = \emptyset$, the parabolic data on $U_{i}$ is trivial and can be represented as
            \[
            (\bm{e}_{i}; \bm{0})_{i} := \left( \{e_{i}^{1}, \ldots, e_{i}^{r} \}; \{0, \ldots, 0\} \right)_{i},
            \]
            where $\{e_{i}^{1}, \ldots, e_{i}^{r} \}$ is a basis for $V|_{U_{i}}$ over $U_{i}$.
        \end{itemize}
        
        The transition isomorphism $g_{ij}$ of $V$ is a morphism of vector bundles with parabolic bases. Specifically:
        \begin{itemize}
            \item If $U_{i} \cap U_{j} \cap D \neq \emptyset$, $g_{ij}$ maps $(V|_{U_{i}}, \bm{e}_{i}, \bm{\alpha})|_{U_{ij}}$ to $(V|_{U_{j}}, \bm{e}_{j}, \bm{\alpha})|_{U_{ij}}$.
            \item If $U_{i} \cap U_{j} \cap D = \emptyset$, $g_{ij}$ is only required to be an isomorphism of vector bundles.
        \end{itemize}
        
        We denote the vector bundle with the parabolic data as
        \[
        \left( V, \{ U_{i} \}_{i \in I}, \{\bm{e}_{i}\}_{i \in I}, \bm{\alpha} \right),
        \]
        or simply $( V, \{\bm{e}_{i}\}_{i \in I}, \bm{\alpha} )$.
    \end{definition}

    We give a more concrete description of the definition \ref{def of global vb with pb}.    	
If $U_{i} \cap D = \emptyset $, the parabolic bases are just the original bases, we write it as
    \begin{equation*}
    \{ f_{1},\cdots,f_{r}\}=\{ e_{1},\cdots,e_{r}\},
    \end{equation*}
    where the latter set is a bases of $V|_{U_{i}}$.            
If $U_{i} \cap D \neq \emptyset$, then the parabolic bases data is 
    \begin{equation*}
    (\{e_{i}^{1},\cdots,e_{i}^{r}\},\{ \alpha_{1},\cdots,\alpha_{r}\}).
    \end{equation*}
    Using a suitable renumbering, we may assume that $\alpha_{i}^{j}$ does not decrease with respect to $j$ for the convenience of exposition.
    And let $s_{i}$ be a section on $U_{i}$ such that $(s_{i})_{0} = D\cap U_{i}$.
    The parabolic bases, by definition, are
\begin{equation*}
\{ f_{i}^{1},\cdots,f_{i}^{r} \} ^{T} = \mathrm{diag}\{ s_{i}^{-\alpha_{1}},\cdots,s_{i}^{-\alpha_{r}} \} \{ e_{i}^{1},\cdots,e_{i}^{r} \}^{T}
\end{equation*}
    if $U_{i}\cap D\neq \emptyset$,
    \begin{equation*}
\{ f_{i}^{1},\cdots,f_{i}^{r} \} ^{T} = \{ e_{i}^{1},\cdots,e_{i}^{r} \}^{T}
\end{equation*}
    if $U_{i}\cap D =  \emptyset$.

Next, we explicitly illustrate the condition for $g_{ij}$ in Definition \ref{def of global vb with pb}.   	
Denote the transition matrix with respect to the bases $\{e_{i}^{\ell}\}$ on $U_{i}$ and the bases $\{ e_{j}^{l}\}$ on $U_{j}$ as $E_{ij} = \mathrm{Mat}(a_{ij})$, \ie, 
\begin{equation*}
g_{ij}\{ e_{i}^{1},\cdots,e_{i}^{r} \}^{T} = E_{ij} \{  e_{j}^{1},\cdots,e_{j}^{r} \}_{\ell=1}^{T}.
\end{equation*}
If $U_{i} \cap U_{j} \cap D \neq \emptyset$, by definition of \ref{def of hom of par b}, $a_{ij}$ is required to be divided by $s_{i}$ (or equivalently $s_{j}$) if $\alpha_{i} > \alpha_{j}$.
    In other words, if we group the weights $\alpha_{l}$ into
    \begin{equation}\label{group of weights}
        \{ \alpha_{1},\cdots,\alpha_{t_{1}}\}_{1}, \cdots, \{ \alpha_{t_{l-1}+1},\cdots, \alpha_{r} \}_{b},
    \end{equation}
    so that the values within the same group are identical.

We write $E_{ij}$ block matrix as
\begin{equation}\label{matrix for ordered bases}
E_{ij} = \begin{Bmatrix}
E_{ij}^{11} & E_{ij}^{12} & \cdots & E_{ij}^{1b}\\
E_{ij}^{21}		   & E_{ij}^{22} & \cdots & E_{ij}^{2b}\\
\vdots  &				& 		& \vdots \\
E_{ij}^{b1} & E_{ij}^{b2} & \cdots & E_{ij}^{bb}
\end{Bmatrix}.
\end{equation}
    based on the grouping of the weights \eqref{group of weights}. 
    Thus $E_{ij}|_{D\cap U_{ij}}$ is a block upper triangular matrix, \ie,
    \begin{equation*}
E_{ij}|_{D\cap U_{ij}} = \begin{Bmatrix}
E_{ij}^{11}|_{D\cap U_{ij}} & E_{ij}^{12}|_{D\cap U_{ij}} & \cdots & E_{ij}^{1b}|_{D\cap U_{ij}}\\
0		   & E_{ij}^{22}|_{D\cap U_{ij}} & \cdots & E_{ij}^{2b}|_{D\cap U_{ij}}\\
\vdots  &				& 		& \vdots \\
0 & 0 & \cdots & E_{ij}^{bb}|_{D\cap U_{ij}}
\end{Bmatrix}.
\end{equation*}
    We conclude the local exposition of Definition \ref{def of global vb with pb} by presenting a comparison of the transition isomorphisms.

    \begin{table}[ht]
        \centering
        \begin{tabular}{|c|c|}
            \hline 
             & Transition data $p_{ij}$ of $V$ \\
            \hline
            $U_{i} \cap D = U_{j} \cap D = \emptyset$ & $E_{ij}$\\
            \hline
            $U_{i} \cap D = \emptyset, \, U_{j} \cap D \neq \emptyset$ & $E_{ij} \mathrm{diag}\{s_{j}^{\alpha_{1}},\cdots,s_{j}^{\alpha_{r}}\}$\\
            \hline
            $U_{i} \cap D \neq \emptyset, \, U_{j} \cap D = \emptyset$ & $\mathrm{diag}\{s_{i}^{-\alpha_{1}},\cdots,s_{i}^{-\alpha_{r}}\}E_{ij}$\\
            \hline
            $U_{i} \cap U_{j} \cap D \neq \emptyset$ & $\mathrm{diag}\{s_{i}^{-\alpha_{1}},\cdots,s_{i}^{-\alpha_{r}}\}E_{ij} \mathrm{diag}\{s_{j}^{\alpha_{1}},\cdots,s_{j}^{\alpha_{r}}\}$\\
            \hline
        \end{tabular}
        \caption{The transition data for higher rank.}
        \label{tab: transition data for higher rank}
    \end{table}

    \begin{remark}\label{equivalence between two datum}
        The transition data for the original bases of the underlying bundle and those for the parabolic bases are mutually convertible. 
        Consequently, knowing the transition data for the parabolic bases allows us to reconstruct the transition data for the original bases. 
        This equivalence enables us to focus on the parabolic bases, which will be shown to be more convenient and to exhibit stronger functorial properties in subsequent discussions.
    \end{remark}

\textbf{Next, we consider the case where $h \geq 1$.}

    Recall the notation $D=D^{1}+\cdots+D^{h}$.
    Let 
    \begin{equation*}
    \bm{\alpha} = (\{ \alpha_{1}^{1},\cdots,\alpha_{r}^{1}), \cdots,(\alpha_{1}^{h},\cdots,\alpha_{r}^{h}\}).
    \end{equation*}
    We only have to modify the definition of parabolic data to a multiindex set.
    The parabolic data on an open subscheme $U_{i}$ is the following set:
    \begin{equation*}
    (\{e_{i}^{1},\cdots,e_{i}^{r}\};\{\alpha_{1}^{t_{1}},\cdots,\alpha_{r}^{t_{1}}\},\cdots,\{\alpha_{1}^{t_{k}},\cdots,\alpha_{r}^{t_{k}}\}),
    \end{equation*}
    where $(e_{i}^{1},\cdots,e_{i}^{r})$ is the bases of the bundle $V|_{U_{i}}$, and 
    \begin{equation}\label{def of index set on Xi}
    T_{i}=\{t_{1},\cdots,t_{k}\} \subset \{ 1,2,\cdots,h\}
    \end{equation}
    is the sub-index set such that $U_{i}$ intersects $D^{t_{i}}$ if and only if $t_{\ell}\in T_{i}$.    
    Let $s_{i,t_{j}}$ be a section on $U_{i}$, such that $(s_{i,t_{j}})_{0}=U_{i} \cap D^{t_{j}}$. 
    Let $\bm{s}_{i}^{\bm{\alpha}_{\ell}}=s_{i,t_{1}}^{\alpha_{\ell}^{t_{1}}}\cdots s_{i,t_{k}}^{\alpha_{\ell}^{t_{k}}}$ for $\ell = 1,2,\cdots,r$.
    The parabolic bases are          
    \begin{equation}\label{explicit parabolic bases for h bigger 1}
   \left\{ e^{1}_{i}/\bm{s}_{i}^{\bm{\alpha}_{1}},\cdots, e^{r}_{i}/\bm{s}_{i}^{\bm{\alpha}_{r}} \right\}.
    \end{equation}

    Similarly, we can define morphisms between free modules with parabolic bases as per Definition \ref{def for morphism for vbp of rank 1}, local parabolic bundles with parabolic bases as described in Definition \ref{def of local pb}, and global parabolic bundles with parabolic bases as outlined in Definition \ref{def of global vb with pb}. 
    Additionally, we can specify the transition data for these parabolic bases.
    Given that the only complication arises from more cumbersome indices without introducing additional conceptual challenges, we omit these details for the sake of clarity.

Next, we define the morphism between vector bundles with parabolic bases.
    Let $(E,\{ U_{i}\}_{i\in I},\{\bm{e}_{i}\}_{i\in I},\bm{\alpha})$ and $(F,\{ U_{j}'\}_{j\in J},\{\bm{f}_{j}\}_{j\in J},\bm{\beta})$ be two vector bundles with parabolic bases.
    Without loss of generality, we may assume that the two open covers coincide.
    
    \begin{definition}\label{def of mor of vb with pb}
    A morphism \( f \) from \( (E, \{\bm{e}_{i}\}_{i \in I}, \bm{\alpha}) \) to \( (F, \{\bm{f}_{i}\}_{i \in I}, \bm{\beta}) \) is an \( \mathcal{O}_{X} \)-morphism \( f: E \to F \) such that, for each open set \( U_{i} \), the restriction of \( f \) to \( U_{i} \) yields a morphism of parabolic bundles 
    \[
    f|_{U_{i}}: (E|_{U_{i}}, \bm{e}_{i}, \bm{\alpha}) \to (F|_{U_{i}}, \bm{f}_{i}, \bm{\beta}),
    \]
    as defined in Definition \ref{def of local morphism for vb with pb}.
    \end{definition}    
    So far, we have defined the category of vector bundles equipped with parabolic bases, which we denote as $\mathcal{PVB}$.    	
Subsequently, we introduce the tensor and $\mathcal{H}om$ within the category $\mathcal{PVB}$.
Let $(E,\{ U_{i}\}_{i\in I},\{\bm{e}_{i}\}_{i\in I},\bm{\alpha})$ and $(F,\{U_{i}'\}_{i\in I},\{\bm{f}_{i}\}_{i\in I},\bm{\beta}) \in \mathcal{PVB}$ be vector bundles with parabolic bases of rank $r_{1}$ and $r_{2}$, respectively.            
As before, we only provide a detailed construction when $h=1$; the case where $h>1$ is a trivial generalization.
    For the index pair $(\ell,m)$ with $\ell \leq r_{1}$, $m \leq r_{2}$, define a set of bases
\begin{equation*}	
(e_{i}^{\ell} \otimes f_{i}^{m})_{0} = e_{i}^{\ell} \otimes f_{i}^{m}/s_{i}^{ [\alpha_{\ell} + \beta_{m}]}
\end{equation*}
    on $U_{i}$, where $[\, \cdot \, ]$ is Gaussian rounding function. 

Then it is easy to check that these local bases $\{\{(e_{i}^{l} \otimes f_{i}^{m})_{0}\}_{\ell,m}\}_{i \in I}$ gives rise to a vector bundle, we denote it by $G$, on $X$.
    We put parabolic data
    \begin{equation}\label{parabolic data of tensor}
        (\{\bm{e} \otimes \bm{f}\}_{i\in I};\{\bm{\alpha} + \bm{\beta}\})=(\{(e_{i}^{\ell} \otimes f_{i}^{m})_{0}\}_{\ell,m;i\in I};\{\{\alpha_{\ell} + \beta_{m}\}\}_{\ell,m})
    \end{equation}
    on $U_{i}$, where $\{ \alpha_{\ell} + \beta_{m} \}$ is the decimal part of $\alpha_{\ell} + \beta_{m}$.
    Then it is elementary, but involved, to check that the transition data satisfies the condition in \ref{def of global vb with pb}.
    Hence \eqref{parabolic data of tensor} are the parabolic data.
    And the corresponding parabolic bases is
    \begin{equation}\label{tensor bases}
        \{(e_{i}^{\ell} \otimes f_{i}^{m})/s_{i}^{\alpha_{\ell}+\beta_{m}}\}_{\ell\leq r_{1},m\leq r_{2}} = \{e_{i}^{\ell}/s_{i}^{\alpha_{\ell}} \otimes f_{i}^{m}/s_{i}^{\beta_{m}}\}_{\ell\leq r_{1},m\leq r_{2}}.
    \end{equation}

    We define the parabolic tensor $(E,\{\bm{e}_{i}\}_{i\in I},\bm{\alpha})\otimes (F,\{\bm{f}_{i}\}_{i\in I},\bm{\beta})$ as 
    \begin{equation*}
    (G,\{\bm{e} \otimes \bm{f}\}_{i\in I};\{\bm{\alpha} + \bm{\beta}\}).
    \end{equation*}
    For the functor $\mathcal{H}om(*,\mathcal{O}_{X})$, we define the bases
\begin{equation*}	
(\check{e_{i}}^{\ell})_{0} = \check{e_{i}}^{\ell}/s_{i}^{[-\alpha_{\ell}]}
    \end{equation*}
    on $U_{i}$.
    These bases give rise to a vector bundle $H$ as well.
    The parabolic data on $U_{i}$ is defined to be 
    \begin{equation*}
        (\bm{\check{e}}_{i\in I};\{-\bm{\alpha}\})=(\{(\check{e_{i}}^{\ell})_{0}\}_{\ell<r_{1},i \in I};\{\{-\alpha_{\ell}\}\}_{\ell \leq r_{1}} ),
    \end{equation*}
and hence the parabolic bases on $U_{i}$ are  
    \begin{equation}\label{bases of Hom}
        \{ (\check{e_{i}}^{1})_{0}/s^{-\{ \alpha_{1}\}},\cdots,(\check{e_{i}}^{r_{1}})_{0}/s^{-\{ \alpha_{r}\}}\} = \{\check{e_{i}}^{1} s_{i}^{\alpha_{1}},\cdots,\check{e_{i}}^{r_{1}} s_{i}^{\alpha_{r}} \}.
    \end{equation}
    We define $\mathcal{H}om((E,\{ U_{i}\}_{i\in I},\{\bm{e}_{i}\}_{i\in I},\bm{\alpha}),\mathcal{O}_{X})=(H,\bm{\check{e}}_{i\in I};\{-\bm{\alpha}\})$.    	
    Finally, we define $\mathcal{H}om(V,W):=\mathcal{H}om(V) \otimes W$ for arbitrary $\mathcal{H}om$.

\begin{remark}
        Initially, the construction may appear intricate; however, it is fundamentally grounded in the simplicity of the parabolic bases, as demonstrated in \eqref{tensor bases} and \eqref{bases of Hom}. 
        These constructions closely resemble those of ordinary vector bundles, making them intuitive and straightforward — a theme that recurs consistently throughout this work.
    \end{remark}

    \begin{remark}  
	Similarly, we can further elaborate on the symmetric product and the determinant for parabolic bundles.   
    \end{remark}

    	\subsection{The correspondence between the category \texorpdfstring{{\boldmath$\mathcal{PVB}$}}{PVB} and the category of parabolic bundles}\label{the section: establising equivalence}

   In this subsection, we demonstrate an equivalence between the category of vector bundles with parabolic bases and the category of abelian parabolic bundles. 
    Let us first revisit the concept of a parabolic sheaf. 
    The principal references we use are \cite{IyerSim07Arelationbetweenparchern} and \cite{KrisSheng20perideRhamovercur}.
    
\begin{definition}\label{def of par bundle}
	A parabolic sheaf on $(X,D)$ is a collection of coherent torsion-free sheaves $F_{\alpha}$ indexed by multiindex $\bm{\alpha} = (\alpha_{1},\cdots,\alpha_{h})$ with $\alpha_{i} \in \mathbb{Q}$, together with inclusions
    \begin{equation*}
        F_{\bm{\alpha}} \hookleftarrow F_{\bm{\beta}}
    \end{equation*}
    whenever $\alpha_{i} \leq \beta_{i}$ (we will write this condition $\bm{\alpha}\leq \bm{\beta}$ shortly), subject to the following hypotheses:
    \begin{itemize}
        \item (normalization) let $\bm{\delta}^{i}$ denote the multi-index $\delta^{i}_{i}=1$, $\delta_{j}^{i} = 0$, $i\neq j$, then $F_{\bm{\alpha}+\bm{\delta}^{i}}=F_{\bm{\alpha}}(-D_{i})$ (compatible with the inclusion); and
        \item (semicontinuity) for any given $\bm{\alpha}$ there exists $c>0$ such that for any multiindex $\bm{\epsilon}$ with $0 \leq \epsilon_{i} <c $ we have $F_{\bm{\alpha}-\bm{\epsilon}} = F_{\bm{\alpha}}$.
    \end{itemize}
\end{definition}

    For each component $D^{i}$ of $D$, we have $F|_{0}(D_{i})=F_{0}/F_{-\delta^{i}}$.
    For any $0 \leq t < 1$, let $F_{D_{i}}$ be the images of $F_{t\delta^{i}}$ in $F|_{0}(D_{i})$, and put
    \begin{equation*}
        \mathrm{Gr}_{D_{i},t} := F_{D_{i},t}/F_{D_{i},t+\epsilon} 
    \end{equation*}
    with $\epsilon$ small.
    From the definition, there are only finite numbers $t$ in $\left[0,1\right)$ such that $\mathrm{Gr}_{D_{i},t}$ is nonzero.   
    These values are called the \emph{weights} of $F$ along the component $D_{i}$.
    We only consider parabolic structures with rational weights.
    
    We are interested in the parabolic sheaf that satisfies the locally abelian condition.
\begin{definition}\label{def of locally abelian par bd}
	A parabolic sheaf $F_{*}$ is called a parabolic bundle if it satisfies the locally abelian condition: in a Zariski neighborhood of any point $p \in X$, there is an isomorphism between $F_{*}$ and a direct sum of parabolic line bundles. 
\end{definition}

    A morphism between two parabolic bundles, represented as $ f: E_{*} \to F_{*} $, is described as a set of maps $ f_{t}: E_{t} \to F_{t} $ for every $ t \in \mathbb{R} $, consistent with the filtration structure (refer to \cite{Yok95InfdefofparHigbd}).        
    We provide a detailed description of the morphism between parabolic line bundles.
    \begin{example}\label{morphism between parabolic line bundle}
    Let $E_{*}$ and $F_{*}$ be two parabolic line bundles with parabolic weights $\alpha$ and $\beta$, respectively.
    
    If $\alpha>\beta$, then by definition, we have the commutative diagram
    \begin{equation*}
    \begin{tikzcd}
        {E_{0}} & {E_{\beta}=E_{0}} & {E_{\alpha}=E_{0}(-D)} \\
{F_{0}} & {F_{\beta}=F_{0}(-D)} & {F_{\alpha}=F_{0}(-D)}
\arrow["{f_{0}}"', from=1-1, to=2-1]
\arrow[hook', from=1-2, to=1-1]
\arrow["{f_{\beta}}"', from=1-2, to=2-2]
\arrow[hook', from=1-3, to=1-2]
\arrow["{f_{\alpha}}", from=1-3, to=2-3]
\arrow[hook', from=2-2, to=2-1]
\arrow[hook', from=2-3, to=2-2]
    \end{tikzcd},
    \end{equation*}
    consequently, it must hold that $D \subset (f_{0})_{0}$, and every other $f_{t}$ is derived from $f_{0}$.
    If $\alpha \leq \beta$, then we have
    \begin{equation*}
        \begin{tikzcd}
        {E_{0}} & {E_{\alpha}=E_{0}(-D)} & {E_{\beta}=E_{0}(-D)} \\
{F_{0}} & {F_{\alpha}=F_{0}} & {F_{\beta}=F_{0}(-D)}
\arrow["{f_{0}}"', from=1-1, to=2-1]
\arrow[hook', from=1-2, to=1-1]
\arrow["{f_{\beta}}"', from=1-2, to=2-2]
\arrow[hook', from=1-3, to=1-2]
\arrow["{f_{\alpha}}", from=1-3, to=2-3]
\arrow[hook', from=2-2, to=2-1]
\arrow[hook', from=2-3, to=2-2]
    \end{tikzcd},
    \end{equation*}
    There are no further constraints on $f_{0}$, and each $f_{t}$ is similarly derived from $f_{0}$.
    \end{example}

    The morphism for parabolic bundles straightforwardly aligns with the one described in Definition \ref{def for morphism for vbp of rank 1}.
    This alignment is crucial for the forthcoming construction.
    The category of parabolic bundles is represented by $\mathrm{Par}\mathcal{VB}$. 
    We will now illustrate the equivalence between the category $\mathrm{Par}\mathcal{VB}$ and the category $\mathcal{PVB}$.

    \textbf{The Functor \(\alpha \colon \mathrm{Par}\mathcal{VB} \to \mathcal{PVB}\)}
    
    Let \(\mathcal{U} = \{U_i\}\) be an open cover of \(X\) such that the restriction of \(V_*\) to each \(U_i\) decomposes as a direct sum of parabolic line bundles:
    \[
    V_*|_{U_i} \cong \bigoplus_{j=1}^r L_{i,*}^j.
    \]
    For each \(i\) and \(j\), let \(e_i^j\) denote a basis of \((L_{i,*}^j)_0\) and \(\alpha_{i,j}^\ell\) denote the parabolic weight of \(L_{i,*}^j\) along the divisor component \(D^\ell\).
    Organize these parabolic weights into an array:
    \[
    \left( 
        \{\alpha_{i,1}^{t_1}, \ldots, \alpha_{i,r}^{t_1}\}; 
        \cdots; 
        \{\alpha_{i,1}^{t_{k}}, \ldots, \alpha_{i,r}^{t_{k}}\} 
    \right),
    \]
    where \(T_i = \{t_1, \ldots, t_k\}\) is the index set defined in \eqref{def of index set on Xi}. 
    By appropriately ordering the parabolic weights, we may assume the array \((\alpha_{i,1}^\ell, \ldots, \alpha_{i,r}^\ell)\) remains constant across all open sets \(U_i\). Consequently, we suppress the index \(i\) and write the array as \((\alpha_1^\ell, \ldots, \alpha_r^\ell)_{\ell=1}^h\).    
    In fact, the data  
    \[
    \left( V_0,\, \{U_i\}_{i \in I},\, \{e_i^1, \ldots, e_i^r\}_{i \in I},\, (\alpha_1^\ell, \ldots, \alpha_r^\ell)_{\ell=1}^h \right)
    \]  
    forms a vector bundle with parabolic bases. To verify this, we analyze the transition data: let  
    \[
    \mathrm{Mat}_{r \times r}(a_{\ell m})_{\ell,m} \quad \text{with } \ell, m = 1, \ldots, r
    \]  
    denote the transition matrix on \(U_{ij}\) relating the bases \(\{e_i^1, \ldots, e_i^r\}\) and \(\{e_j^1, \ldots, e_j^r\}\). The critical condition -- that \(s_{i,t}\) divides \(a_{\ell m}\) whenever \(\alpha_\ell^t > \alpha_m^t\) -- follows directly from the compatibility of parabolic structures under base change, as demonstrated in Example \ref{morphism between parabolic line bundle}.  
    
    For morphisms: let \(f \colon V_* \to W_*\) be a morphism of parabolic bundles. By an identical argument, the underlying map \(f_0\) satisfies the compatibility conditions in Definition~\ref{def of mor of vb with pb}, thereby inducing a morphism of vector bundles with parabolic bases.

    \textbf{The Functor \(\beta \colon \mathcal{PVB} \to \mathrm{Par}\mathcal{VB}\)}

    Given a vector bundle with parabolic bases \((V, \{U_i\}, \bm{e}_i, \bm{\alpha})\), we construct a parabolic bundle as follows. Assume \(h = 1\) without loss of generality. On each \(U_i\), where the parabolic data is \((\{e_i^1, \ldots, e_i^r\}, \{\alpha_1, \ldots, \alpha_r\})\), define parabolic line bundles \(L_{i,*}^\ell\) by:
    \begin{itemize}
        \item Letting \((L_{i,*}^\ell)_0\) be the free \(\mathcal{O}_{U_i}\)-module generated by \(e_i^\ell\);
        \item Assigning parabolic weight \(\alpha_\ell\) along \(D\).
    \end{itemize}
    
    To glue these locally defined parabolic bundles into a global parabolic bundle \(V_*\), we define subsheaves \(\{V_t\}_{t \in \mathbb{R}}\) in \(V\) satisfying Definition~\ref{def of par bundle}. Using notation from \eqref{group of weights} and \eqref{matrix for ordered bases}, for \(\alpha_{t_j} < \beta \leq \alpha_{t_j+1}\) and \(X_k \cap D \neq \emptyset\), choose local bases on \(X_k\):
    \[
    \{f_k^{1,\beta}, \ldots, f_k^{r,\beta}\} = \left\{ \frac{e_i^1}{s_k}, \ldots, \frac{e_i^{t_j}}{s_k}, e_i^{t_j+1}, \ldots, e_i^r \right\}.
    \]
    The transition matrix between \(U_i\) and \(U_k\) becomes:
    \[
    F_{ik}^\beta = \mathrm{diag}\Big(\underbrace{s_i, \ldots, s_i}_{t_j}, 1, \ldots, 1\Big) \cdot E_{ik} \cdot \mathrm{diag}\Big(\underbrace{s_k^{-1}, \ldots, s_k^{-1}}_{t_j}, 1, \ldots, 1\Big),
    \]
    where \(E_{ik}|_D\) is block upper triangular. Since \(F_{ik}^\beta \in \mathrm{Mat}_{r \times r}(\mathcal{O}(U_{ik}))^*\), the sheaves \(V^\beta\) glue canonically to form a parabolic bundle \(V_*\). Define \(\beta(V) \coloneqq V_*\), extending to \(h > 1\) by repeating along each divisor component.
    
    For morphisms \(g \colon (V, \bm{e}_i, \bm{\alpha}) \to (W, \bm{f}_i, \bm{\beta})\), the maps \(\Fil_t V \to \Fil_t W\) inherit parabolic compatibility through the same gluing process. The functors \(\alpha\) and \(\beta\) form a categorical equivalence: \(\alpha\) maps parabolic bundles to those with parabolic bases, while \(\beta\) reconstructs the parabolic structure. That \(\alpha\) and \(\beta\) are quasi-inverses follows directly from the construction.
    
In the last part of this section, we show that the equivalence preserves the $\otimes$ and $\mathcal{H}om$ functor.
Recall the definition of tensor and $\mathcal{H}om$ functor in the category of parabolic bundles.
\begin{definition}[Definition of $\mathcal{H}om$]\label{def of hom of par b}
	Let $\mathrm{Hom}(E_{*},\mathcal{O}_{X})$ denote the set of parabolic morphisms. 
	We let $\mathcal{H}om(E_{*},\mathcal{O})$ denote the sheaf of homomorphisms defined by taking $\mathrm{Hom}(E_{*},\mathcal{O})(U) = \mathrm{Hom}(E_{*}|U,\mathcal{O}|U)$. 
	We also define a parabolic sheaf $\mathcal{H}om(E_{*},\mathcal{O})_{*}$ by taking $\mathrm{Hom}(E_{*},\mathcal{O})_{\alpha} = \mathrm{Hom}(E_{*},\mathcal{O}_{\alpha})$ and $i^{\alpha,\beta}_{Hom(E_{*},\mathcal{O})}$ for $\alpha \geq \beta$ to be the natural maps induced by 
    \begin{equation*}
        i^{[\alpha,\beta]}_{\mathcal{O}}: \mathcal{O}[\alpha]_{*} \to \mathcal{O}[\beta]_{*}.
    \end{equation*}
\end{definition}

\begin{example}\label{examp of par hom}
	Let $L_{*}$ be a parabolic line bundle.
	Then, by the definition of a parabolic bundle, we claim that $L^{\vee}_{*}=\mathcal{H}om(L_{*},\mathcal{O})_{*}$ can be described as follows: $L^{\vee}_{0} = L^{\vee}(-D)$, and $L^{\vee}_{1-\alpha} = L^{\vee}$.    		
\end{example}   	
\begin{definition}[Definition of tensor]\label{def of tensor of par}
	Let $\tau: X-D \hookrightarrow X$ denote the natural inclusion.
	Suppose $V_{*}$ and $W_{*}$ are two parabolic bundles; define $(V_{*}\otimes W_{*})_{\alpha} = \sum_{\alpha_{1}+\alpha_{2}=\alpha}V_{\alpha_{1}}\otimes W_{\alpha_{2}}$, seen as a subbundle of $\tau_{*}\tau^{*}(E \otimes F)$. 
	We take $i^{\alpha,\beta}_{V_{\alpha}\otimes W_{\beta}}$ as the natural inclusion map.
\end{definition}

\begin{example}\label{examp of par tensor}
	Let $L_{*}$ and $M_{*}$ be two parabolic line bundles.
	If $\alpha_{1}+\alpha_{2} < 1$, the tensor product is simpler.
	$(L_{*}\otimes M_{*})_{t} = L\otimes M$, if $0 \leq t \leq \alpha_{1}+\alpha_{2}$; $(L_{*}\otimes M_{*})_{t} = L\otimes M(-D)$, if $t > \alpha_{1}+\alpha_{2}$.    		
	A special case is if $\alpha_{1}+\alpha_{2}\geq1$.
	$(L_{*}\otimes M_{*})_{t} = L\otimes M(D)$, if $0 \leq t \leq \alpha_{1}+\alpha_{2}-1$ $(L_{*}\otimes M_{*})_{t} = L\otimes M$, if $t > \alpha_{1}+\alpha_{2}-1$.
	
\end{example}
        
\begin{proposition}\label{prop of equi preserves Hom and tensor}
	The functors $\alpha$ and $\beta$ are compatible with $\otimes$ and $\mathcal{H}om$.
\end{proposition}
\begin{proof}
	We only have to prove it locally.
	Choose an open neighborhood $U$, such that both parabolic bundles $V_{*}$ and $W_{*}$ split as parabolic line bundles, \ie, 
	\begin{equation*}
	V_{*}|_{U}\cong \oplus V_{*}^{i}, \quad W_{*}|_{U} \cong \oplus W_{*}^{j}.
	\end{equation*}
	Then by the definition of a tensor product,
	\begin{equation*}
	V_{*} \otimes W_{*}|_{U} \cong \oplus V_{*}^{i} \otimes W_{*}^{j}, 
	\end{equation*}
	and hence
	\begin{equation*}
	(V_{*} \otimes W_{*})_{0}|_{U} = \oplus (V_{*}^{i} \otimes W_{*}^{j})_{0}.
	\end{equation*}
	Then by the computation of the example \ref{examp of par tensor}:
	If $\alpha_{i} + \beta_{j} < 1$, then $(V_{*}^{i} \otimes W_{*}^{j})_{0} \cong V \otimes W$, the corresponding filtration weights are $\alpha_{i} + \beta_{j}$.
	If $\alpha_{i} + \beta_{j} \geq 1$, then $(V_{*}^{i} \otimes W_{*}^{j})_{0} \cong V \otimes W(D)$, the corresponding filtration weight is $\alpha_{i} + \beta_{j} - 1$.
	
	Comparing the parabolic structure with the parabolic data \ref{parabolic data of tensor} yields our desired outcome. 
	  Using the same approach, we can demonstrate that the result holds for the $\mathcal{H}om$ functor.
		
\end{proof}

    The results of this section are encapsulated in the subsequent theorem.
    \begin{theorem}\label{thm of equivalence between two category}
        Let notations be as above. The category of vector bundles with parabolic bases $\mathcal{PVB}/(X,D)$ is tensor equivalent to the category of parabolic bundles $\mathrm{Par}\mathcal{VB}/(X,D)$.
    \end{theorem}
    
    In accordance with the theorem outlined above, we will refer to the parabolic bases of a parabolic bundle $V_{*}$ as that of $\alpha(V_{*})$ for simplicity.    
    The definitions for parabolic degree and stability within the category of vector bundles equipped with a parabolic structure have been omitted. 
    These can be directly inferred from the definitions pertaining to parabolic bundles, as detailed by Yokogawa \cite{Yok93Comofmodofparbd}.

\section{Local theory of parabolic \texorpdfstring{{\boldmath$\lambda$}}{lambda}-connections}\label{Local theory of par lam con}

\subsection{Parabolic \texorpdfstring{{\boldmath$\lambda$}}{lambda}-connections}	
We extend the equivalence discussed in the previous section to incorporate the category of $\lambda$-connections. 
\begin{definition}\label{def of lam con}
	Suppose $\lambda\in \Gamma(\mathcal{O}_{X},X)$, a $\lambda$-connection is a pair $(H,\nabla)$, such that $H$ is a vector bundle, and $\nabla:H \to H\otimes \Omega_{X}^{1}(\log D)$ is a $k$-linear morphism satisfying the Leibniz rule
	\begin{equation}
	\nabla(fe) = f\nabla(e)+ \lambda \dif f \otimes e,
	\end{equation}
    where $f\in \mathcal{O}$, and $e\in H$.
\end{definition}

\begin{definition}\label{def of par lam connection}
        A \textit{parabolic $\lambda$-connection} over \((X,D)\) consists of a pair \((V_{*}, \nabla)\), where:
        \begin{itemize}
            \item \(V_{*}\) is a parabolic vector bundle on \(X\);
            \item \(\nabla: V_{*} \to V_{*} \otimes \Omega_X^1(\log D)\) is a $\lambda$-connection that preserves the parabolic structure. 
        \end{itemize}
        
        More precisely, the compatibility condition requires:
        \begin{itemize}
        \item For every parabolic weight \(\bm{\alpha}\), the restriction \(\nabla_{\bm{\alpha}}: V_{\bm{\alpha}} \to V_{\bm{\alpha}} \otimes \Omega_X^1(\log D)\) is a $\lambda$-connection;
        \item For \(\bm{\alpha} > \bm{\beta}\), the diagram
        \[
        \begin{tikzcd}
            V_{\bm{\alpha}} \arrow[r, "\nabla_{\bm{\alpha}}"] \arrow[d, hook] & V_{\bm{\alpha}} \otimes \Omega_X^1(\log D) \arrow[d, hook] \\
            V_{\bm{\beta}} \arrow[r, "\nabla_{\bm{\beta}}"] & V_{\bm{\beta}} \otimes \Omega_X^1(\log D)
        \end{tikzcd}
        \]
        commutes, where the vertical arrows are the natural inclusion maps.
        \end{itemize}
        
        The parabolic $\lambda$-connection is called \textit{flat} if \(\nabla_{\bm{\alpha}}\) is flat (i.e., its curvature vanishes) for every parabolic weight \(\bm{\alpha}\).
    \end{definition}

    Taking a simple analysis analogue to \ref{morphism between parabolic line bundle} indicates that $\mathrm{Res}_{D_{i}}$ preserves $V_{D_{i},t}|_{D_{i}}$ for all $t \in [0,1)$.
The parabolic $\lambda$-connections of primary interest to us are referred to as adjusted parabolic $\lambda$-connections and strong parabolic $\lambda$-connections.
    \begin{definition}\label{def of adjusted lam connection}
        Let $(V_{*}, \nabla)$ be a parabolic $\lambda$-connection defined on a pair $(X, D)$. 
        We shall refer to this connection as adjusted if the residue $\mathrm{Res}_{D_{i}}\nabla$, when restricted to $\mathrm{Gr}_{D_{i},\alpha_{i}}V$, has eigenvalues $\lambda \alpha_{i}$. 
        Furthermore, if $\mathrm{Res}_{D_{i}}\nabla$ acts on $\mathrm{Gr}_{D_{i},\alpha_{i}}V$ as $\lambda \alpha_{i} \mathrm{Id}_{D_{i}}$, where $\mathrm{id}_{D_{i}}$ denotes the identity transformation, then the connection is termed a strong parabolic $\lambda$-connection.
    \end{definition}
    \begin{remark}\label{remark: stong parabolic}
        When $\lambda=1$, the strong parabolic connection aligns with the one defined as the strong parabolic connection in \cite[Definition 4.3]{NielsAmine23Parconandrootstack}.
    \end{remark}

Then we define the $\lambda$-connection with parabolic bases.
    Let \((V, \{U_i\}_{i \in I}, \{\bm{e}_i\}_{i \in I}, \bm{\alpha})\) be a vector bundle equipped with parabolic bases, and let \(\nabla\) be a \(\lambda\)-connection on \(V\). 
    On each open subscheme \(U_i\), the connection \(\nabla\) is expressed in terms of the parabolic basis \(\{\bm{e}_i\} = \{e_i^1, \ldots, e_i^r\}\) as
    \begin{equation}\label{eq of local connection}
        \nabla\{\bm{e}_i\}^\mathsf{T} = M_i \cdot \bm{e}_i^\mathsf{T},
    \end{equation}
    where \(M_i = (m_{jk})_{r \times r}\) is an \(r \times r\) matrix with entries in \(\Omega_X^1(\log D)|_{U_i}\). 
    The connection \(\nabla\) is called \textbf{admissible with respect to the parabolic structure along \(D^i\)} if it satisfies:
    \begin{itemize}
        \item For all indices \(j, k\) with \(\alpha_j > \alpha_k\), the residue condition \(\mathrm{Res}_{D^i}(m_{jk}) = 0\) holds.
    \end{itemize}
    One readily verifies that if \(\nabla\) is admissible on an open subscheme \(U_\ell\) along \(D^i \cap U_\ell\), then admissibility is preserved on any open subscheme \(U \subseteq X\) containing \(D^i \cap U\).
    
    \begin{definition}\label{def of connection with par bases}
        A \textit{\(\lambda\)-connection with parabolic bases} is a tuple
        \begin{equation*}
            \Big(V,\, \{U_i\}_{i \in I},\, \{\bm{e}_i\}_{i \in I},\, \bm{\alpha},\, \nabla\Big),
        \end{equation*}
        where:
        \begin{itemize}
            \item \((V, \{U_i\}_{i \in I}, \{\bm{e}_i\}_{i \in I}, \bm{\alpha})\) is a vector bundle with parabolic bases;
            \item \(\nabla\) is a \(\lambda\)-connection on \(V\) that is admissible along every irreducible component \(D^i\) of \(D\).
        \end{itemize}
    \end{definition}
   A \textit{morphism of $\lambda$-connections with parabolic bases} \( f \colon (V, \{\bm{e}_i\}, \bm{\alpha}, \nabla) \to (W, \{\bm{f}_i\}, \bm{\beta}, \nabla') \) is a morphism \( f \colon (V, \{\bm{e}_i\}, \bm{\alpha}) \to (W, \{\bm{f}_i\}, \bm{\beta}) \) of parabolic vector bundles that additionally satisfies the $\lambda$-connection compatibility condition:  
   \[
   \nabla' \circ f = (f \otimes \mathrm{id}) \circ \nabla.
   \]
    
    We compute the connection matrix \( M_{\mathrm{para}} \) of \( \nabla \) in Definition~\ref{def of connection with par bases} with respect to the parabolic bases. This matrix is derived via a base change applied to the original connection matrix \( M \), combined with the Leibniz rule, resulting in functorial behavior.
    
    Let \( T_i \subseteq \{1, \ldots, h\} \) denote the subset of indices \( t \) for which \( D^t \cap U_i \neq \emptyset \). For each \( t \in T_i \), choose a defining section \( s_i^t \) of \( D^t \) on \( U_i \). The parabolic bases on \( U_i \) are then given by:
    \[
    \left\{ \frac{e_i^1}{\bm{s}_i^{\bm{\alpha}_1}}, \ldots, \frac{e_i^r}{\bm{s}_i^{\bm{\alpha}_r}} \right\},
    \]
    where \( \bm{s}_i^{\bm{\alpha}_j} = \prod_{t \in T_i} (s_i^t)^{\alpha_j^t} \).
    The transformed connection matrix \( M' \) is computed as:
    \begin{equation}
    M' = \mathrm{diag}\left( \bm{s}_i^{\bm{\alpha}_1}, \ldots, \bm{s}_i^{\bm{\alpha}_r} \right) \cdot M \cdot \mathrm{diag}\left( \bm{s}_i^{-\bm{\alpha}_1}, \ldots, \bm{s}_i^{-\bm{\alpha}_r} \right),
    \end{equation}
    where \( M \) is the original connection matrix defined in \eqref{eq of local connection}.
    
    By the Leibniz rule, the action of \( \nabla \) on the parabolic bases is:
    \begin{equation}\label{fml of base change}
    \begin{aligned}
    \nabla \{ \frac{e_i^1}{\bm{s}_i^{\bm{\alpha}_1}}, \ldots, \frac{e_i^r}{\bm{s}_i^{\bm{\alpha}_r}} \}^\mathsf{T} = ( & -\lambda \, \mathrm{diag}\left( \bm{\alpha}_1 \dif \log \bm{s}_i, \ldots, \bm{\alpha}_r \dif \log \bm{s}_i \right) \\
    & + M' ) \{ \frac{e_i^1}{\bm{s}_i^{\bm{\alpha}_1}}, \ldots, \frac{e_i^r}{\bm{s}_i^{\bm{\alpha}_r}} \}^\mathsf{T},
    \end{aligned}
    \end{equation}
    where for each \( j \),
    \[
    \bm{\alpha}_j \dif \log \bm{s}_i = \sum_{t \in T_i} \alpha_j^t \dif \log s_i^t.
    \]
    The matrix \( M_{\mathrm{para}} \coloneqq -\lambda \, \mathrm{diag}\left( \bm{\alpha}_1 \dif \log \bm{s}_i, \ldots, \bm{\alpha}_r \dif \log \bm{s}_i \right) + M' \) is the connection matrix relative to the parabolic bases. The \textit{parabolic residue} of \( \nabla \) along \( D^i \) is defined as the restriction:
    \[
    \mathrm{Res}_{D^i}(M_{\mathrm{para}}) \coloneqq M_{\mathrm{para}} \big|_{D^i}.
    \]
    \begin{definition}\label{def of nilpotent residue par conn}		
	A $\lambda$-connection with parabolic bases is called adjusted (respectively strong parabolic) when the parabolic residues of $\nabla$ at each $D^{i}$ are nilpotent (respectively zero).
\end{definition}

We extend Theorem \ref{thm of equivalence between two category} to apply to the context of $\lambda$-connections as follows:
\begin{theorem}\label{thm of equiv between adjusted par and nil conn with par bas}
        Let $(X,D)$ be as before.
        The category of parabolic $\lambda$-connections on $(X,D)$ is tensor equivalent to the category of $\lambda$-connections with parabolic bases on $(X,D)$. 
        
        Furthermore, under this tensor equivalence, the adjusted parabolic $\lambda$-connections are in one-to-one correspondence with the adjusted $\lambda$-connections with parabolic bases.
    \end{theorem}

\begin{proof}
        Without loss of generality, it suffices to prove the proposition for \( h = 1 \). Let \( V_* \) be a parabolic bundle and \( \nabla_0 \) a \( \lambda \)-connection on \( V_0 \).
        
        We aim to show that \( \nabla_0 \) underlies a parabolic connection on \( V_* \) if and only if \( \nabla_0 \) is admissible along \( D \) with respect to the parabolic weights \( \bm{\alpha}(V_*) \). This equivalence follows directly from the commutative diagram in Example~\ref{morphism between parabolic line bundle}, as the compatibility conditions are formally identical.

        Consider an adjusted parabolic \( \lambda \)-connection. Locally, choose parabolic bases 
        \[
        \{ e^1, \ldots, e^r \}
        \]
        with associated parabolic weights \( \alpha_1 \leq \cdots \leq \alpha_r \). By definition, the connection matrix of \( \nabla_0 \) relative to these bases is:
        \begin{equation*}
        \nabla_0 \big\{ e^1, \ldots, e^r \big\}_D^\mathsf{T} 
        = \Big( \lambda \, \mathrm{diag}\{ \alpha_1, \ldots, \alpha_r \} + N \Big) \dif(\log s) \big\{ e^1, \ldots, e^r \big\}_D^\mathsf{T},
        \end{equation*}
        where \( N|_D \) is upper triangular and nilpotent. 
        
        The residue of \( \nabla_0 \) with respect to the parabolic bases \( \{ e^i / s^{\alpha_i} \}_i \) is computed as:
        \[
        \left. \mathrm{diag}\big\{ s^{-\alpha_1}, \ldots, s^{-\alpha_r} \big\} \cdot N \cdot \mathrm{diag}\big\{ s^{\alpha_1}, \ldots, s^{\alpha_r} \big\} \right|_D,
        \]
        which remains nilpotent due to the upper triangular structure of \( N|_D \).
        The converse follows analogously by reversing the argument. Finally, verifying that the operations \( \otimes \) and \( \mathcal{H}om \) preserve the \( \nabla \)-structure reduces to a local computation, which is straightforward using the functoriality of connections and the Leibniz rule.
\end{proof}

    As in the parabolic bundle case, we will refer to the parabolic bases of a parabolic $\lambda$ -connection $(V_{*},\nabla)$ as that of $(\alpha(V_{*}),\nabla)$.

\subsection{Pullback and Pushforward of Parabolic \texorpdfstring{{\boldmath$\lambda$}}{lambda}-connections}

    The parabolic pullback is typically described using the root stack method, as illustrated by Iyer and Simpson \cite{IyerSim07Arelationbetweenparchern} and Ajneer and Sheldon \cite{DhiJoy12pullbackofparbdandcover}. 
    However, Alfaya and Biswas \cite{AlBis23pullandpshofparbd}, Biswas and Machu \cite{BIS19ondirectimageofparvbandparcon}, and Kumar and Majumder \cite{Kumar18parabdinpositivechar} have proposed a more streamlined approach for constructing parabolic pullbacks in the context of finite morphisms between algebraic curves. 
    We will extend this approach by employing parabolic bases to construct the pullback and pushforward of  parabolic $\lambda$-connections over varieties of arbitrary dimension.
    \subsubsection{Pullback of parabolic \texorpdfstring{{\boldmath$\lambda$}}{lambda}-connection}
    Let \( f \colon X \to Y \) be a flat morphism of smooth varieties over \( k \), and let \( D \hookrightarrow Y \) be a normal crossings divisor. Write \( D = \sum_{s=1}^h D^s \), and define the reduced pullback divisors:
    \[
    B = (f^* D)_{\mathrm{red}}, \quad B^s = (f^* D^s)_{\mathrm{red}}.
    \]
    Each \( B^s \) decomposes into irreducible components \( B^s = \sum_t B^{s,t} \). There exist positive integers \( n^{s,t} \) such that:
    \[
    f^* D^s = \sum_t n^{s,t} B^{s,t}.
    \]

    Let \( (H_*, \nabla) \) be a parabolic \(\lambda\)-connection on \( (X, D) \) with parabolic data:
    \[
    \left( \{ U_i \}_{i \in I}, \{\bm{e}_i\}_{i \in I}, \bm{\alpha} \right).
    \]
    For each \( U_i \), let \( \bm{x}_i = \{ x_i^{t_1}, \ldots, x_i^{t_k} \} \) be local sections satisfying \( (x_i^{t_\ell})_0 = U_i \cap D^{t_\ell} \). The parabolic bases on \( U_i \) are given by:
    \[
    \left\{ \frac{\bm{e}_i}{\bm{x}_i^{\bm{\alpha}}} \right\}.
    \]
    Let \( V_{i,j} \subseteq Y \) be an open affine subscheme such that \( f(V_{i,j}) \subseteq U_i \). Choose local sections \( \{ y_{i,j}^{s,t} \} \) on \( V_{i,j} \) satisfying:
    \[
    (y_{i,j}^{s,t})_0 = V_{i,j} \cap B^{s,t}.
    \]
    By the definition of \( B^{s,t} \), the pullback sections satisfy:
    \[
    x_i^{t_\ell} = \prod_t (y_{i,j}^{t_\ell, t})^{n^{\ell, t}},
    \]
    which we abbreviate as:
    \[
    x_i^{t_\ell} = \left( \bm{y}_{i,j}^{t_\ell} \right)^{\bm{n}^{t_\ell}}.
    \]
    Here, bold notation denotes multi-index products over the components of \( B \).
            
    Define a formal bases on $V_{ij}$ as 
    \begin{equation*}
    f^{*}\bm{e}_{i}/f^{*}(\bm{x}_{i}^{\bm{\alpha}})=f^{*}\bm{e}_{i}/\bm{y}_{i,j}^{\bm{n}\bm{\alpha}},
    \end{equation*}
    where $\bm{x}_{i}^{\bm{\alpha}}=\Pi_{\ell} (x_{i}^{t_{\ell}})^{\alpha_{t^{\ell}}}  $, $\bm{y}_{i,j}^{\bm{n}\bm{\alpha}}=\Pi_{\ell} (\bm{y}_{i,j}^{t_{\ell}})^{\bm{n}^{t_{\ell}}\alpha_{t^{\ell}}} $.
    The intricate expression, while structurally complex, fundamentally represents a straightforward concept that can be concisely summarized as a base change or, in other terms, a change of variables.
    Hence the parabolic data on $Y_{i,j}$ is 
    \begin{equation}\label{parabolic data of pullback}
        (f^{*}\bm{e}_{i}/\bm{y}_{i,j}^{[\bm{n}\alpha]}, \{ \bm{n}\bm{\alpha} \}).
    \end{equation}
    By the same pattern, we can show that these local bases $\{f^{*}\bm{e}_{i}/\bm{t}_{i,j}^{[\bm{n}\bm{\alpha}]}\}_{i,j}$ are naturally bonded to a global bundle, which we write as $(f^{*}_{\para}(H_{*}))_{0}$.
    Consequently, we obtain a parabolic bundle, denoted as $f^{*}_{\para}(H_{*})$, whose parabolic data are
    \begin{equation*}
        ((f^{*}_{\para}(H_{*}))_{0},\{ Y_{ij}\},f^{*}\bm{e}_{i}/\bm{y}_{i,j}^{[\bm{n}\bm{\alpha}]}, \{ \bm{n}\bm{\alpha} \}).
    \end{equation*}
    Then we consider the connection part, suppose 
    \begin{equation*}
        \nabla_{\bm{y}\frac{\partial}{\partial \bm{y}}}(\bm{e}_{i}/\bm{y}_{i}^{\bm{\alpha}}) =  \bm{e}_{i}/\bm{y}_{i}^{\bm{\alpha}}\bm{m},
    \end{equation*}
    where $\bm{m}$ is a matrix valued in $ \mathcal{O}_{U_{i}}$, then define a canonical connection $f^{*}_{\para}\nabla$ on $f^{*}H_{*}$ by setting:
    \begin{equation*}
        f^{*}_{\para}\nabla_{\bm{y}\frac{\partial}{\partial \bm{y}}}(f^{*}(\frac{\bm{e}_{i}}{\bm{y}_{i,j}^{\bm{n}\bm{\alpha}}})) = f^{*}_{\para}\nabla_{\bm{y}\frac{\partial}{\partial \bm{y}}} (\frac{f^{*}(\bm{e}_{i})}{\bm{x}_{i,j}^{\bm{n}\bm{\alpha}}}) f^{*}\bm{m} = f^{*}_{\para}\nabla_{\bm{x}\frac{\partial}{\partial \bm{x}}} (\frac{f^{*}(\bm{e}_{i})}{\bm{x}_{i,j}^{\bm{n}\bm{\alpha}}}) \bm{n} f^{*}\bm{m},
    \end{equation*}
    where $\bm{x}^{\bm{n}}=\bm{y}$.
    It is easy to check that $f^{*}_{\para}\nabla$ is a well-defined parabolic $\lambda$-connection on the parabolic bundle $f^{*}_{\para}H_{*}$.
    A specific type of pullback arises when \( f \) is finite and the ramification index of \( B^{s,t} \) are the common multiples of the denominators of parabolic weights along \(D^{s} \). 
    From Equation \eqref{parabolic data of pullback}, it follows that the parabolic structure of \( f^*H_{*} \) along \( D^{s} \) is trivial. 
    More precisely, on each component \( Y_{ij} \), we have
    \[
    f^*_{\para}H_{*} \cong \bigoplus_{j=1}^{r} f^*e^{j}_{i} \mathcal{O}_{Y_{ij}} \otimes \mathcal{O}_{Y_{ij}}(\bm{n}\bm{\alpha}(B \cap Y_{ij})).
    \]
    This construction aligns with the parabolic pullback defined by Biswas in \cite{Bisw88parbdasorbibd}.
We encapsulate the aforementioned construction in the following proposition.
	
\begin{proposition}\label{prop of pullback of bd with p}
Let the notation be as above.
    There is a natural pullback functor \( f^*_{\mathrm{para}} \) from the category of parabolic \(\lambda\)-connections on \( (Y, B) \) to the category of parabolic \(\lambda\)-connections on \( (X, D) \), mapping \( (H_*, \nabla) \) to \( (f^*_{\mathrm{para}} H_*, f^*_{\mathrm{para}} \nabla) \). 
    Locally, this functor is characterized by pulling back parabolic bases to parabolic bases.
\end{proposition}	

We provide a significant instance of a parabolic bundle pullback, extending a well-known classical result related to the Frobenius pullback.
\begin{proposition}\label{prop of Fro of par bd}
        Suppose $\mathrm{char}(k)=p$, and $H_{*}$ is a parabolic bundle on $(X,D)$, which is defined over $k$.
        The Frobenius pullback $F^{*}_{\para}$ is endowed with a canonical parabolic connection $\nabla_{\mathrm{can}}$. 		
\end{proposition}	
\begin{proof}
   We only need to verify the construction locally. Let \( \{ e_1 / s^{\alpha_1}, \ldots, e_r / s^{\alpha_r} \} \) denote the parabolic bases. Applying the parabolic pullback, we obtain the following parabolic data:
    \[
    \left( \{ F^* e_1 / s^{[p \alpha_1]}, \ldots, F^* e_r / s^{[p \alpha_r]} \}; \{ \{ p \alpha_1 \}, \ldots, \{ p \alpha_r \} \} \right),
    \]
    where \( [p \alpha_i] \) and \( \{ p \alpha_i \} \) denote the integer and fractional parts of \( p \alpha_i \), respectively.
    A natural connection \( \nabla \) is induced on the parabolic bases. For \( f \in \mathcal{O}_X \), it acts as:
    \[
    \nabla(f \otimes F^* e_i / s^{p \alpha_i}) = \dif f \otimes F^* e_i / s^{p \alpha_i}.
    \]
    The connection \( \nabla_0 \) acts on the bases \( \{ F^* e_1 / s^{[p \alpha_1]}, \ldots, F^* e_r / s^{[p \alpha_r]} \} \) as follows:
    \[
    \nabla(F^* e_1 / s^{[p \alpha_1]}) = \nabla\left( s^{\{ p \alpha_1 \}} \cdot F^* e_1 / s^{p \alpha_1} \right) = \{ p \alpha_1 \} \frac{\dif s}{s} \otimes F^* e_1 / s^{[p \alpha_1]}.
    \]
    By definition, this connection is exactly the adjusted connection.
\end{proof}
    
    \subsubsection{Pushforward of parabolic \texorpdfstring{{\boldmath$\lambda$}}{lambda}-connection}\label{def of pushforward of par}     
    We construct this pushforward for finite branched coverings using parabolic bases.
    A morphism $f : X \to Y$ of normal varieties is termed a branched covering if it is a finite and surjective morphism, as outlined in \cite{Ulud07Orbifoldanduniformization} and \cite{Mak87brachedcoveandalgfunc}.     
    Our attention is directed towards the situation in which $X$ and $Y$ are smooth varieties. 
    Additionally, consider $G$ as the automorphism group associated with $f$. 
    We presume that the ramification divisor $D=\sum D_{\ell}$ and its image $B=f(D)$ exhibit normal crossings.
    We denote the above data concisely as $f:(X,D) \to (Y,B)$.
    For any $p \in X(\Bar{k})$, let $q=f(p)$.
    Take \'etale coordinates $(x_{1},\cdots,x_{d})$ at $p$ and $(y_{1},\cdots,y_{d})$ at $q$ such that the divisor $D'$ and $B'$ are defined by equations $x_{i}=0$ and $y_{j}=0$ for some $i$ and $j$, respectively.
    By the definition of branched covering, $f^{\#}_{p}: \mathcal{O}_{Y,q} \to \mathcal{O}_{X,p}$ sends $y_{i}$ to $x_{i}^{n_{i}}$ for some positive integer $n_{i}$.
   
    Let $(E_{*},\nabla)$ be a parabolic $\lambda$-connection on $(X,D+D')$. 
    Without loss of generality, we may assume that $D$ is $G$-equivariant and contains the ramification divisor.
    Since we can take the composition 
    \begin{equation*}
    \nabla':H_{*} \stackrel{\nabla}{\to} H_{*}\otimes \Omega^{1}_{X}(\log (D+D')) \hookrightarrow H_{*} \otimes \Omega_{X}^{1}(\log (G(D)+D')),
    \end{equation*}
    where $G(D)$ is the $G$ saturation of $D$. 
    We only have to replace $(H_{*},\nabla)$ with $(H_{*},\nabla')$.    
    The essential step to construct the pushforward functor is to perform a local construction, generalizing Example \ref{example: toy example of pushforward}.    
    Let $U$ (resp. $V$) be an open affine scheme of $p$(resp. $q$), such that $f(U)=V$ and $H_{*}$ are direct sums of the free parabolic line bundles.
    Let $H_{*} \cong \oplus_{j} L_{*}^{j}$ be a direct sum of parabolic line bundles on $U$, such that each $L_{*}^{j}$ is endowed with a canonical bases $ e^{j} $, then $(f|_{U})_{*}(L_{*}^{j})$ is a bundle with bases
    \begin{equation*}
        \{e^{j} x_{1}^{t_{1}}\cdots x_{d}^{t^{d}} \}_{0\leq t_{1} \leq n_{1},\cdots ,0\leq t_{d} \leq n_{d}}.
    \end{equation*}
    Suppose the parabolic bases of $e^{j}$ are $e^{j}/(x_{1}^{\alpha_{1}^{j}}\cdots x_{n}^{\alpha_{n}^{j}})$, we set the parabolic bases for $e^{j} x_{1}^{t_{1}}\cdots x_{d}^{t^{d}}$ as
    \begin{equation*}
        \{(e^{j} x_{1}^{t_{1}}\cdots x_{d}^{t^{d}}/(y_{1}^{\alpha^{j}_{1}/n_{1}+t_{1}/n_{1}}\cdots y_{n}^{\alpha^{j}_{n}/n_{d}+t_{d}/n_{d}})\}_{j;\, t_{1},\cdots , t_{d}}.
    \end{equation*}
    The corresponding parabolic datum is 
    \begin{equation*}
        (e^{j} x_{1}^{t_{1}}\cdots x_{d}^{t_{d}};\{\frac{\alpha^{j}_{1}+t_{1}}{n_{1}},\cdots,\frac{\alpha^{j}_{d}+t_{d}}{n_{d}}\} )_{j;\, t_{1},\cdots , t_{d}}.
    \end{equation*}

    It is easy to check that these local parabolic bundles are glued to a global parabolic bundle, denoted by $f_{\para,*}H_{*}$ as before.
    The connection, denoted as $f_{\para,*}\nabla$, is naturally induced on $f_{\para,*}H_{*}$.
    Suppose 
    \begin{equation*}
    \nabla_{x_{i}\frac{\partial}{\partial x_{i}}}\{ \bm{e}_{i}/\bm{x}_{i}^{\bm{\alpha}} \}=\bm{m}_{i}\{ \bm{e}_{i}/\bm{x}_{i}^{\bm{\alpha}} \}.
    \end{equation*}
    The connection matrix of 
    \begin{equation*}
    f_{\para,*}\nabla_{y_{i}\frac{\partial}{\partial y_{i}}}\{ \bm{e}_{i}x^{t_{i}}/\bm{y}_{i}^{(\bm{\alpha}+t_{i})/n} \}=\frac{1}{n}\bm{m}_{i}\{ \bm{e}_{i}x^{t_{i}}/\bm{y}_{i}^{(\bm{\alpha}+t_{i})/n} \}
    \end{equation*}
    in terms of the parabolic bases.
    In summary, the connection matrix of the parabolic bases is derived by multiplying $1/n$.
    The connection for $(f_{\para,*}H_{*})_{0}$ is thus determined from the pattern in the last section.
    \begin{example}
        Let the notations be as above and let $d=1$, $n_{1}=2$, and $\mathrm{rank}{(H_{*})}=1$, with parabolic weight $\alpha$.
        Suppose that the connection on $U$ is
        \begin{equation*}
            \nabla_{*}(e/x^{\alpha})=m e/x^{\alpha} \dif \log x,
        \end{equation*}
        where $m \in \mathcal{O}_{U}$.
        Then the connection acts on the parabolic bases on $Y$ as follows: 
        \begin{equation*}
            f_{*}\nabla_{*}
            \begin{Bmatrix}
                \frac{e}{y^{\alpha/2}} \\
                \frac{ex}{y^{(\alpha+1)/2}} \}
            \end{Bmatrix} = \begin{Bmatrix}
                \frac{m}{2} & 0\\
                0 & \frac{m}{2}
            \end{Bmatrix}\begin{Bmatrix}
                \frac{e}{y^{\alpha/2}} \\
                \frac{ex}{y^{(\alpha+1)/2}} \}
            \end{Bmatrix}\dif \log y.
        \end{equation*}
    \end{example}
    \begin{remark}
    Our definition aligns with \cite{BIS19ondirectimageofparvbandparcon} and \cite{AlBis23pullandpshofparbd} when $f$ is a morphism between Riemann surfaces, as verified through local calculations and the comparison of the resulting parabolic weights.      
    \end{remark}
    
    So far, we have defined the pullback (for flat morphism) and pushforward (for finite morphism) of parabolic connections. 
We have the following proposition, which generalizes the BIS correspondence \cite[Theorem 5.8]{KrisSheng20perideRhamovercur}.
    \begin{theorem}\label{thm of Bis for vb with pb}
        If $f: (X,D) \to (Y,B)$ is a finite branched Galois covering between smooth varieties over $k$, such that the ramification indices are all non-zero in $k$.
        Then we have the following equivalence between two categories:
        \begin{equation*}
        \begin{tikzcd}
        \left\{ \vcenter{\hbox{ $  G$-equivariant parabolic}\hbox{\qquad $\lambda$-connections}}  \right\} \Bigm/(X,D)  & \{\textit{parabolic }\lambda\textit{-connections}\}/(Y,B)
	\arrow["{(f_{\para,*})^{G}}"{pos=0.6}, shift left, harpoon, from=1-1, to=1-2]
	\arrow["{f^{*}_{\mathrm{par}}}", shift left, harpoon, from=1-2, to=1-1]
        \end{tikzcd}.
        \end{equation*}
        Moreover, the $G$-equivariant adjusted (strong) parabolic $\lambda$-connections on $(X,D)$ correspond to the adjusted (strong) parabolic $\lambda$-connections on $(Y,B)$ under the equivalence.
    \end{theorem}
    
    Before we give proof, we present two lemmas.

    \begin{lemma}\label{lem of functoriality of G structure}
        Let $(V,\nabla)$ be a $\lambda$-connection on $(Y,B)$, then $f_{\para}^{*}(V,\nabla)$ is a $G$-equivariant $\lambda$-connection.
        Conversely, let $(H,\nabla')$ be a $G$-equivariant connection on $(X,D)$, then $G$ acts on the bundle $f_{\para,*}(H,\nabla')$. 
    \end{lemma}
    \begin{proof}
        Without loss of generality, we can take $H$ to be free and set $h=1$, where the parabolic data is given by 
        \begin{equation*}
            \{ e_{1},\ldots,e_{r} ; \alpha_{1},\ldots,\alpha_{r} \}.
        \end{equation*}
        Pick sections $x$ and $y$ on $X$ and $Y$ satisfying $x_{0}=D$ and $y_{0}=B$ respectively, and further assume $x^{n}=y$.
        Then the parabolic bases of $(H,\nabla)$ is
        \begin{equation*}
            \{ e_{1}/y^{\alpha_{1}},\ldots,e_{r}/y^{\alpha_{r}} \},
        \end{equation*}
        and the parabolic bases of $\pi_{\para}^{*}(H,\nabla)$ is
        \begin{equation*}
            \{ f^{*}e_{1}/x^{n\alpha_{1}},\ldots,f^{*}e_{r}/x^{n\alpha_{r}} \},
        \end{equation*}
        the parabolic data is
        \begin{equation*}
            \{ f^{*}e_{1}/x^{[n\alpha_{1}]}, \ldots,f^{*}e_{r}/x^{[n\alpha_{r}]};\{ n\alpha_{1}\},\ldots,\{n\alpha_{r}\} \}.
        \end{equation*}
        The bases $x^{[n \alpha_{i}]}$ naturally possesses a $G$-equivariant structure, which similarly applies to the parabolic bases when a trivial $G$-equivariant structure is imposed on $f^{*}e_{i}$.
        The connection can be naturally induced.
        
        On the contrary, let $(H,\nabla)$ be a $G$-equivariant parabolic connection on $(X,D)$, suppose the parabolic bases of $(H,\nabla)$ is the form
        \begin{equation*}
            \{ e_{1}, \ldots, e_{r};\alpha_{1},\cdots,\alpha_{r} \},
        \end{equation*}
        and each $e_{i}$ is $G$-equivariant.
        By definition, the parabolic bases are
        \begin{equation*}
            \{ e_{i}/y^{\alpha_{i}/n}, e_{i}x/y^{\alpha_{i}/n+1/n}, \ldots, e_{i}x^{n-1}/y^{\alpha_{i}/n+(n-1)/n}\}_{i},
        \end{equation*}
        note that $G$ canonically acts on each $e_{i}x^{j}$, hence $G$ acts on $\pi_{\para,*}(H,\nabla)$. 
    \end{proof}
    
    \begin{lemma}[Projection Formula]\label{projection formula}
        $f_{\para,*} f_{\para}^{*}(H,\nabla') = (H,\nabla') \otimes f_{\para,*}(\mathcal{O}_{X},\lambda \dif)$, where $(\mathcal{O}_{X},\lambda \dif)$ is the canonical $\lambda$-connection on $(X,D)$.
    \end{lemma}
    \begin{proof}
        The demonstration arises from analyzing the parabolic bases on both sides.
    \end{proof}
    Both Lemma \ref{lem of functoriality of G structure} and Lemma \ref{projection formula} are well-established in the context of $\lambda$-connections with trivial parabolic structure.  
    We now turn to the proof of Theorem \ref{thm of Bis for vb with pb}.
    
    \begin{proof}[Proof of Theorem \ref{thm of Bis for vb with pb}]
        By Lemma \ref{lem of functoriality of G structure}, the functors $\pi_{\para,*}$ and $\pi_{\para}^{*}$ are well-defined.
        We only have to prove two natural isomorphisms
        \begin{equation}\label{back then decent}
        (f_{\para,*})^{G} f_{\para}^{*} \cong \mathrm{Id} 
        \end{equation}
        and 
        \begin{equation}\label{decent then back}
        f_{\para}^{*} f_{\para,*} \cong \mathrm{Id}.
        \end{equation}
        By lemma \ref{projection formula}, to prove \eqref{back then decent} we only have to show $(f_{\para,*})^{G}(\mathcal{O}_{X},\lambda \dif) \cong (\mathcal{O}_{Y},\lambda \dif)$, where $(\mathcal{O}_{X},\lambda \dif)$ is endowed with trivial $G$-equivariant structure, which is obvious.

        To prove \eqref{decent then back}, we only have to show it in the case of line bundle, $e/t^{\alpha}$ be the $G$-equivariant parabolic bases; then the parabolic bases of $f_{\para,*}H_{*}$ is
        \begin{equation*}
            \{ \frac{e}{y^{\alpha/n}}, \cdots, \frac{ex^{n-1}}{y^{(\alpha+n-1)/n}} \}.
        \end{equation*}
        Since $e$ is endowed with a structure $G$-equivariant, there exists a unique natural number $k$ less than $n$ such that $ex^{k}$ is $G$-invariant.
        
        Therefore, the parabolic bases of $(f_{\para,*}(H_{*}))^{G}$ is given by $\frac{ex^{k}}{y^{(\alpha+k)/n}}$. 
        The pullback for the parabolic bases $\frac{ex^{k}}{y^{(\alpha+k)/n}}$ equals $\frac{ex^{k}}{x^{(\alpha+k)}}=\frac{e}{x^{\alpha}}$.
        The adjusted (strongly) condition is straightforward because the residue of a parabolic bases behaves functorially when subjected to pullback and pushforward operations.
    \end{proof}
        
    \begin{remark}
        As noted in Proposition \ref{prop of pullback of bd with p}, when restricted to $G$-equivariant $\lambda$-connections with a trivial parabolic structure on the left, this aligns with the Biswas-Iyer-Simpson correspondence.
    \end{remark}
    \begin{remark}
        The theorem is also established in \cite[Propositions 1.46 and 1.51]{LiuYangZuo23parabolibFFmoandparHiggsdeRhamflows} for the case of curves, based on their characterization of parabolic structures.
    \end{remark}

\section{Parabolic Cartier and Parabolic Inverse Cartier transformation}\label{Par Higgs-de Rham flow}
The foundational contributions to non-abelian Hodge theory in positive characteristic were made by Ogus and Vologodsky \cite{OV2007NonAbeHincharp}, and Schepler further built on their work by addressing the logarithmic case \cite{Sche08LognonabeHodincharp}. 
    Later, Lan, Sheng, and Zuo utilized exponential twist methods, as detailed in \cite{LSZ15NonabeHodexptwi} and the logarithmic case \cite[Appendix]{LSYZ19uniofpadiccurveviaflow}, to reestablish this correspondence. 
    In this paper, we employ the constructed parabolic bases to formulate the parabolic non-abelian Hodge correspondence in positive characteristic using the exponential twist technique.
    
    \subsection{Revisit of exponential twist}    
    
    \subsubsection{Inverse Cartier Transformation via Exponential Twist}
    Accordingly, we begin by providing a concise overview of the exponential twist construction.
    Let $(E,\theta)$ be a logarithmic Higgs bundle on $(X,D)$.
Choose an open cover $\{ U_{i} \}_{i\in I}$ for $X$ such that each $U_{i}$ has a Frobenius lifting $\tilde{F}_{i}$ that respects $U_{i} \cap D$, and $E|_{U_{i}}$ is a free module.
    We select bases $\{ e_{1}^{i}, \ldots, e_{r}^{i} \}$ for $E|_{U_{i}}$, and label the transition matrix on $U_{ij}$ from the bases $\{ e_{1}^{i}, \ldots, e_{r}^{i} \}$ to $\{ e_{1}^{j}, \ldots, e_{r}^{j} \}$ as $\phi_{ij}$.

Locally, on each $U'_{i}=F^{*}U_{i}$, we define the flat bundle $(H,\nabla)_{i}$ as $(F^{*}E|_{U_{i}},\nabla_{i})$, where $F^{*}$ is the Frobenius pullback and the associated connection on $F^{*}E|_{U_{i}}$ is defined by 
    \begin{equation*}
    \nabla_{i} = \nabla_{\mathrm{can},U_{i}} + \frac{\dif \tilde{F}^{*}_{i}}{p} F^{*}\theta_{i},
    \end{equation*}
    where $\nabla_{\mathrm{can},U'_{i}}$ is the canonical connection on $F^{*}E|_{U_{i}}$, and
    \begin{equation*}
        \frac{\dif \tilde{F}^{*}_{i}}{p}: F^{*}\Omega_{X}^{1}(\log D) \to \Omega_{X'}^{1}(\log D).
    \end{equation*}
    It is direct to check each $(H,\nabla)_{i}$ is a flat connection with nilpotence exponent $\leq p$.
    The exponential twist is a technique to glue these $(F^{*}E|_{U_{i}},\nabla_{i})$ into a global connection.
    Note that the morphism
    \begin{equation*}
        \frac{\tilde{F}_{i}^{*}-\tilde{F}_{j}^{*}}{p}: \mathcal{O}_{U_{ij}} \to \mathcal{O}_{U_{ij}},
    \end{equation*}  
    factors through the differential in the sense that there is a homomorphism 
    \begin{equation}\label{def of DI class}
        h_{ij}: F^{*}\Omega_{X}^{1}(\log D)|_{U_{ij}} \to \mathcal{O}_{X'}^{1}|_{U'_{ij}},
    \end{equation}
    on $U_{ij}$.
    The collection $\{h_{ij}\}_{i,j}$ gives rise to a C\v{e}ch class $\kappa$ in $H^{1}(X,F^{*}T_{X}(-\log D))$, known as the Deligne-Illusie class. 
    This class is independent of the chosen cover, but only dependent on the $W_{2}(k)$ lifting $(X,D) \hookrightarrow (\tilde{X},\tilde{D})$.
    
The key step is to define the transition matrix $\varphi_{ij}$ from $\{ F^{*}e_{\ell}^{i} \}_{\ell=1,\cdots,r}$ to $\{ F^{*}e_{\ell}^{j} \}_{\ell=1,\cdots,r}$ as
    \begin{equation*}
        \exp(h_{ij}F^{*}(\theta_{ij})) \circ F^{*}\phi_{ij}
    \end{equation*}
    on $U'_{ij}$.
    It can then be confirmed that these data satisfy the cocycle condition and are compatible with $\nabla_{i}$.
    Consequently, this yields a globally defined connection $(H, \nabla)$ on $(X', D')$. 
    It is then a standard procedure to verify that $(H, \nabla)$ is independent of the chosen cover.

    In summary, we conclude the construction $C^{-1}$ into the following diagram: 
    \begin{equation}\label{exp con of C^{-1}}
        \begin{tikzcd}
	& {(F^{*}E,\nabla_{\mathrm{can}},F^{*}\theta)} \\
	{(H,\nabla)} && {(E,\theta)}
	\arrow["{F^{*}}"'{pos=0.5}, from=2-3, to=1-2]
	\arrow["{\text{exp. twist}}"'{pos=0.5}, from=1-2, to=2-1]
        \end{tikzcd}
    \end{equation}
                
    \subsubsection{Cartier Transformation via Exponential Twist.}
    For a flat bundle $(H,\nabla)$ with $p$-curvature of nilpotent exponent less than $p-1$ on $(X',D')$.
    
    Let notations be as above, to initiate the construction of $C(H,\nabla)$, employ an exponential twist to construct a new flat bundle as described:
    Define a new connection on $H_{i}$ by the formula
    \begin{equation*}
        \nabla'_{i} = \nabla_{i} + \frac{\dif \tilde{F}^{*}_{i}}{p}(\psi_{i}),
    \end{equation*}
    it can be checked that $\psi_{i}$ is parallel with respect to $\nabla'_{i}$, and the $p$-curvature of $\nabla'_{i}$ vanishes.
    Then define the transition data as
    \begin{equation*}
        \exp(h_{ij}\psi) \circ F^{*}\varphi_{ij},
    \end{equation*}
    these data satisfy the cocycle condition and are compatible with $\psi_{i}$'s.            
    Hence we obtain a triple $(H',\nabla',\psi')$, satisfying the properties:
    \begin{itemize}
        \item $(H',\nabla')$ is a flat connection with vanishing $p$-curvature;
        \item $\psi': H' \to H' \otimes F^{*}\Omega_{X}^{1}(\log D)$ is an $\mathcal{O}_{X}$-linear morphism that arises naturally from the $p$-curvature $\psi(\nabla)$ and is parallel under $\nabla'$.
    \end{itemize}
    
    Then by applying the Cartier descent theorem, we get a vector bundle $E$ on $X$; moreover, it can be shown that the morphism $\psi$ also descends on a Higgs field on $E$.
    The main thing to be checked is that the triple $(H',\nabla',\psi')$, 
    As a consequence, applying the Cartier descent theorem, we obtain a vector bundle $E$ with a $\mathcal{O}_{X}$-morphism $\theta$, which serves as a Higgs field on $E$.
    The Cartier transformation is defined to be the composition of exponential twist and Cartier descent functors as above, as depicted in the following diagram:
    \begin{equation}\label{exp con of C}
        \begin{tikzcd}
	& {(H',\nabla',\psi')} \\
	{(H,\nabla)} && {(E,\theta)}
	\arrow["{\text{Cartier descent}}"{pos=0.5}, from=1-2, to=2-3]
	\arrow["{\text{exp. twist}}"{pos=0.5}, from=2-1, to=1-2]
        \end{tikzcd}
    \end{equation}
    
```

\subsection{Parabolic inverse Cartier and Cartier Transformation}
The core idea behind constructing the \textit{parabolic inverse Cartier transform} lies in systematically adapting the classical exponential twist framework to the parabolic setting: one takes parabolic bases and the parabolic Higgs field as foundational input data, applies the exponential twist mechanism, and rigorously verifies that the output intrinsically carries the structure of a parabolic connection.
This procedure follows rigorously in parallel to the original exponential twist framework.

```
    Prior to formalizing this construction, we adopt notation from \cite{KrisSheng20perideRhamovercur} and define:
    \begin{itemize}
        \item $\mathrm{HIG}_{lf,*}^{p-1}(X,D)$: Category of parabolic Higgs bundles on $(X,D)$ where:
        \begin{itemize}
            \item Higgs fields are nilpotent of degree $\leq p-1$;
            \item Residues along each $D^i$ are nilpotent of degree $\leq p-1$.
        \end{itemize}
        \item $\mathrm{MIC}_{lf,*}^{p-1}(X,D)$: Category of flat bundles on $(X,D)$ where:
        \begin{itemize}
            \item $p$-curvatures are nilpotent of degree $\leq p-1$;
            \item Residues along each $D^i$ are nilpotent of degree $\leq p-1$.
        \end{itemize}
    \end{itemize}
\subsubsection{Parabolic Inverse Cartier Transformation}
For a parabolic Higgs bundle \((E_*, \theta)\) and an open covering \(\{U_i\}\) of \(X\), there exist for each \(U_i\) parabolic bases \(\left\{ e_i^1/s_i^{\alpha_1}, \ldots, e_i^r/s_i^{\alpha_r} \right\}\) and a local Frobenius lift \(F_i\), and denote the Deligne-Illusie data by $h_{ij}$ on the intersections $U_{ij}$.

    Let $(E_{*},\theta)$ be an object in $\mathrm{HIG}_{lf,*}^{p-1}$.
Define a formal connection on $U_{i}$ as $(H_{*,i},\nabla_{i})$, where $H_{*,i}$ is formally generated by the bases 
    \begin{equation}\label{par bases of inverse Cartier}
    \{F^{*}(e_{i}^{j}/s_{i}^{\alpha_{j}}) \} = \{ F^{*}e_{i}^{j}/s_{i}^{p_{\alpha_{j}}}\},
    \end{equation}
    and the connection is defined to be 
\begin{equation}\label{local par connection}
\nabla_{\para,i} = \nabla_{\mathrm{can},i} + \frac{\dif \tilde{F}^{*}_{i}}{p} \circ F^{*}\theta_{\para,i},
\end{equation}
    where $\nabla_{\mathrm{can}}$ is defined by Proposition \ref{prop of Fro of par bd}.
Our first claim is that $(H_{*},\nabla)_{i}$ is endowed with a parabolic connection structure.
    The parabolic data are as follows.
    \begin{equation}\label{local par data of inversee Cartier}
        (\{F^{*}e_{i}^{1}/ s_{i}^{[p\alpha_{1}]},\cdots,F^{*}e_{i}^{r}/ s_{i}^{[p\alpha_{r}]}\},\{\{p\alpha_{1}\},\cdots,\{p\alpha_{r}\}\}),
    \end{equation}
    where $[*]$ is the Gaussian rounding function and $\{*\}$ is the decimal function. 
    The parabolic bases are \eqref{par bases of inverse Cartier}.
    The connection matrix for the parabolic bases \ref{par bases of inverse Cartier} is directly derived from the Frobenius pullback of the Higgs matrix associated with the parabolic bases $\left\lbrace e_{i}^{1}/s_{i}^{\alpha_{1}},\ldots,e_{i}^{r}/s_{i}^{\alpha_{r}}\right\rbrace $, composed with $\frac{\dif \tilde{F}^{*}_{i}}{p}$, thus fulfilling the adjusted condition.

    Then we glue these local parabolic connections $(H,\nabla)_{i}$ to a global connection $(H,\nabla)$.
    We define the transition matrix for the parabolic bases \eqref{par bases of inverse Cartier} on $U_{ij}$ as 
    \begin{equation}\label{glueing data for par bases}
        \exp(h_{ij}F^{*}\theta_{\para})F^{*}(\phi_{ij,\mathrm{par}}).
    \end{equation} 
    The essential verifications include: the flatness of $\nabla_{i}$, the cocycle condition of the transition data, and the well-defined nature of the global connection. 
    Given that our construction mirrors that of \cite{LSZ15NonabeHodexptwi}, the validation procedure is identical and thus will not be reiterated here.             
    We demonstrate that the derived formal connection is indeed an adjusted parabolic connection.
    \begin{proposition}\label{prop of par inverse Cartier}
        Let $(E_{*},\theta)$ be an object in the category $\mathrm{HIG}^{p-1}_{lf,*}$.
        Then $(H_{*},\nabla):=C^{-1}_{\exp}(E_{*},\theta)$ is an object in the category $\mathrm{MIC}_{lf,*}^{p-1}$.
\end{proposition}
\begin{proof}
        We have two things to check: the formal data above give rise to a parabolic connection bundle, and it satisfies the adjusted condition.
	Since we have shown that $(H_{*},\nabla)_{i}$ is a parabolic connection, it remains to check the transition data.
        Let $T=\mathrm{diag}\{s_{i}^{\alpha_{1}},\cdots,s_{i}^{\alpha_{r}}\}$, by the formula 
        \begin{equation*}
            \exp(h_{ij}F^{*}\theta_{\para}) = \exp(h_{ij}F^{*}T^{-1}F^{*}\theta F^{*}T) = F^{*}T^{-1}\exp(h_{ij}F^{*}\theta )F^{*}T.
        \end{equation*}

        Hence, \eqref{glueing data for par bases} is indeed a well-defined parabolic transition data for the parabolic bases \eqref{par bases of inverse Cartier}.
	Checking the adjusted condition is equivalent to showing that the residue of \eqref{local par connection} is nilpotent; the latter is obvious.
    \end{proof}
            
    \subsubsection{Parabolic Cartier Transformation}
    We now construct the Cartier morphism in this subsection.
    Before that, we have to extend the classical Cartier descent theorem \cite[Theorem 5.1]{Katz70Nilpotneconnections} to the parabolic context.
    \begin{theorem}[Cartier Descent Theorem]\label{thm of Cartier Descent}
        Let the notations remain as above.  
        The following two categories are equivalent:
        \begin{equation*}
        \begin{tikzcd}
        \{  \text{Parabolic vector bundle } V'_{*} \} /(X',D')  & \left\{ \vcenter{\hbox{Strong parabolic $\lambda$-connection $(V_{*},\nabla)$} \hbox{\qquad with vanishing $p$-curvature}}  \right\} \Bigm/(X,D)
	\arrow["{F^{*}}"{pos=0.6}, shift left, harpoon, from=1-1, to=1-2]
	\arrow["{( \, )^{\nabla}}", shift left, harpoon, from=1-2, to=1-1]
        \end{tikzcd},
        \end{equation*}
        where \( F^{*} \) denotes the Frobenius pullback defined in Proposition \ref{prop of Fro of par bd}, and \( ( \, )^{\nabla} \) is the functor that assigns horizontal sections of parabolic bases.
    \end{theorem}
    Our goal is to show that $( \, )^{\nabla}$ naturally induces a parabolic structure and serves as a quasi-inverse to the Frobenius pullback.  
    The proof follows the original argument of Katz \cite{Katz70Nilpotneconnections}, adapted here to the parabolic setting.
    \begin{proof}
    Let \((V, \nabla)\) be a parabolic connection with vanishing \(p\)-curvature. To verify this property, it suffices to work locally. Let \(s_{1}, \dots, s_{n}\) be a system of \'etale coordinates. It is then clear that the parabolic residues of \((V_{*}, \nabla)\) all vanish.  
    Let \(\nabla_{i} = \nabla_{\frac{\partial}{\partial s_{i}}}\). We define the \(f^{-1}(\mathcal{O}_{X'})\)-linear operator  
    \begin{equation*}
    \mathrm{P} = \sum_{w=0}^{p-1} \prod_{i=1}^{d} \frac{(-s_{i})^{w}}{w!} \prod_{i=1}^{d} (\nabla_{i})^{(w)},
    \end{equation*}
    as introduced in \cite[Theorem 5.1]{Katz70Nilpotneconnections}.
    Since \((\frac{\partial}{\partial s_{i}})^{p}\) and \(\psi(\nabla) = 0\), we have  
    \begin{equation*}
    (\nabla_{i})^{(p)} = \nabla_{(\frac{\partial}{\partial s_{i}})^{p}} - (\nabla_{i})^{(p)} = \psi(\nabla)\left(\frac{\partial}{\partial s_{i}}\right) = 0.  
    \end{equation*}
    Thus, \(\mathrm{P}\) satisfies the equation  
    \begin{equation*}
    \nabla \mathrm{P} = 0  
    \end{equation*}  
    by direct computation, which implies the following properties:  
    \begin{itemize}  
        \item \(\mathrm{P}(V) \subset V^{\nabla}\);  
        \item \(\mathrm{P}|_{V^{\nabla}} = \text{id}\);  
        \item \(\mathrm{P}^{2} = \mathrm{P}\) is a projection onto \(V^{\nabla}\);  
        \item the intersection \(\bigcap_{w=0}^{p-1} \ker\left(\mathrm{P} \prod_{i=1}^{d} \nabla_{i}^{w}\right) = 0\).  
    \end{itemize}                            
    Next, we examine the action of \(\mathrm{P}\) on the parabolic bases. 
    We claim that this action induces a system of parabolic bases on \(X'\).  
    Suppose \(\alpha^{j}_{i} = m^{j}_{i}/n^{j}_{i}\) is the irreducible representation of \(\alpha^{j}_{i}\), and let \(\ell^{j}_{i}\) be the smallest integer such that \(p\) divides \((m_{i}^{j} + \ell^{j}_{i} n_{i}^{j})\).
    Since \(\nabla\) is \(\mathcal{O}_{X'}\)-linear, we compute:
    \[
      \mathrm{P}\left(\frac{e_i}{s_j^{\alpha_i^j}}\right) = \mathrm{P}\left(\frac{e_i s_j^{\ell_i^j}}{s_j^{\alpha_i^j + \ell_i^j}}\right) = \frac{\mathrm{P}(e_i s_j^{\ell_i^j})}{s_i^{\alpha_i^j + \ell_i^j}}.
    \]
    Let \(\beta_i^j \coloneqq \frac{\alpha_i^j + \ell_i^j}{p}\). Then:
    \[
      \mathrm{P}\left(\frac{e_i}{s_i^{\alpha_i^j}}\right) = \frac{\mathrm{P}(e_i s_i^{\ell_i^j})}{F^*\left((s'_i)^{\beta_i^j}\right)},
    \]
    which implies that \(\left\{\mathrm{P}\left(\frac{e_i}{s_j^{\alpha_j}}\right)\right\}\) forms a parabolic basis on \(X'\) with parabolic weights \(\beta_i^j\).
    The parabolic data is
    \begin{equation*}
        \{\mathrm{P}(e_{i}s_{j}^{\ell_{i}^{j}}) ;\beta_{i}^{j} \}_{i,j}.
    \end{equation*}                
    Subsequently, the inverse of the natural morphism $V^{\nabla} \otimes \mathcal{O}_{X} \to V$ can be constructed utilizing $P$ as
    \begin{equation*}
    \begin{split}
        T: V &\to V^{\nabla} \otimes \mathcal{O}_{X}, \\
           e &\mapsto \sum_{w} \Pi_{i=1}^{d}\frac{s_{i}^{w}}{w!}\mathrm{P} \,\Pi_{i=1}^{d}\nabla_{i}^{w}(e). 
    \end{split}
    \end{equation*}
    This completes the proof.
    \end{proof}
\begin{remark}
    The above theorem is also established by Wakabayashi \cite{Wak24Frobeniuspullbackanddormantopers}[Theorem 4.5], who generalizes a classical result of Cartier on higher-level Frobenius descent to the parabolic framework for algebraic curves.
\end{remark}
    
    Denote the category of parabolic connections with vanishing $p$-curvature as $\mathcal{PC}on_{0}$.
    \begin{proposition}\label{Cartier equivalence commutes with pullback and pushforward}
        Let $f: (X,D) \to (Y,B)$ be a branched covering between logarithmic curves, then pullback and pushforward commute with Cartier descent and Frobenius pullback, \ie, we have the following two natural isomorphisms,
        \begin{equation*}
        \begin{tikzcd}
	{\{\mathcal{PC}on_{0}\}/(X,D)} & {\{\mathcal{PVB}\}/(X',D')} & {\{\mathcal{PC}on_{0}\}/(X,D)} & {\{\mathcal{PVB}\}/(X',D')} \\
	{\{\mathcal{PC}on_{0}\}/(Y,B)} & {\{\mathcal{PVB}\}/(Y',B'),} & {\{\mathcal{PC}on_{0}\}/(Y,B)} & {\{\mathcal{PVB}\}/(Y',B').}
	\arrow[shift left, harpoon, from=1-1, to=1-2]
	\arrow[shift left, harpoon, from=1-2, to=1-1]
	\arrow[shift left, harpoon, from=1-3, to=1-4]
	\arrow["{f_{*}}"', from=1-3, to=2-3]
	\arrow[shift left, harpoon, from=1-4, to=1-3]
	\arrow["{f_{*}}", from=1-4, to=2-4]
	\arrow["{f^{*}}", from=2-1, to=1-1]
	\arrow[shift left, harpoon, from=2-1, to=2-2]
	\arrow["{f^{*}}"', from=2-2, to=1-2]
	\arrow[shift left, harpoon, from=2-2, to=2-1]
	\arrow[shift left, harpoon, from=2-3, to=2-4]
	\arrow[shift left, harpoon, from=2-4, to=2-3]
        \end{tikzcd}
        \end{equation*}
        Furthermore, when $f$ is Galois, the parabolic descent commutes with both Cartier descent and Frobenius as well, \ie, the following diagram commutes
        \begin{equation*}
            \begin{tikzcd}
            	{\{\mathcal{PC}on_{0}\}/(X,D)} & {\{\mathcal{PVB}\}/(X',D')} \\
            	{\{\mathcal{PC}on_{0}\}/(Y,B)} & {\{\mathcal{PVB}\}/(Y',B').}
            	\arrow["{(\cdot)^{\nabla}}", shift left, from=1-1, to=1-2]
            	\arrow["{f_{*}^{G}}", shift left, from=1-1, to=2-1]
            	\arrow["{\otimes \mathcal{O}_{X}}", shift left, from=1-2, to=1-1]
            	\arrow["{f_{*}^{G}}", shift left, from=1-2, to=2-2]
            	\arrow["{f^{*}}", shift left, from=2-1, to=1-1]
            	\arrow["{(\cdot)^{\nabla}}", shift left, from=2-1, to=2-2]
            	\arrow["{f^{*}}", shift left, from=2-2, to=1-2]
            	\arrow["{\otimes \mathcal{O}_{Y}}", shift left, from=2-2, to=2-1]
            \end{tikzcd}
        \end{equation*}
    \end{proposition}
    \begin{proof}
        The proof is based on the analysis of the parabolic bases.
        By the definition of $f^{*}$, for any section $e$, we have 
        \begin{equation*}
        (f^{*}\nabla)f^{*}e=f^{*}\nabla e, 
        \end{equation*}
        hence $\mathrm{P}_{f^{*}(\nabla)}(f^{*}e)=f^{*}\mathrm{P}(e)$.
        By the definition of $f_{*}$, we have 
        \begin{equation*}
        \mathrm{P}\left(\frac{es^{m}}{t^{m/n}}\right)=\frac{\mathrm{P}(es^{m})}{t^{m/n}},
        \end{equation*}
        since 
        \begin{equation*}
        (f_{*}\nabla)\left(\frac{es^{m}}{t^{m/n}}\right)=\frac{\nabla(es^{m})}{t^{m/n}}.
        \end{equation*}
        Hence, we have the proof of the commutative diagram.
        The functoriality of the Frobenius pullback can be shown in the same pattern.
        
        Assuming $f$ is a Galois cover, we affirm the commutativity of the parabolic descent between $f^{*}$ and $f_{*}$ employing an analogous argument to the one described in \ref{thm of Bis for vb with pb} regarding the parabolic bases.               
    \end{proof}
            
    Subsequently, we develop the parabolic Cartier transformation. 
    Let $(H_{*},\nabla)\in \MIC_{lf,*}^{p-1}$, the Cartier transformation of $(H,\nabla)$, denoted as $C_{\exp}(H,\nabla)$, is succinctly displayed in the diagram below:
    \begin{equation}\label{para exp con of C}
        \begin{tikzcd}
	& {(H'_{*},\nabla',\psi')} \\
	{(H_{*},\nabla)} && {(E_{*},\theta)}
	\arrow["{\text{par. Cartier descent}}"{pos=0.5}, from=1-2, to=2-3]
	\arrow["{\text{par. exp. twist}}"{pos=0.5}, from=2-1, to=1-2]
        \end{tikzcd},
    \end{equation}
    where the parabolic exponential twist remains consistent with \eqref{exp con of C}, substituting the input data with a parabolic bases and its related connection and transition data, while the parabolic Cartier descent refers to the functor as demonstrated in Theorem \ref{thm of Cartier Descent}.

    \begin{proposition}\label{prop of par Cartier}
        Let $(H_{*},\nabla)$ be an object in the category $\mathrm{MIC}^{p-1}_{lf,*}$.
        Then $(E_{*},\theta):=C_{\exp}(H_{*},\nabla)$ is an object in the category $\mathrm{HIG}_{lf,*}^{p-1}$.
\end{proposition}
    \begin{proof}
        The proof is exactly the same as the initial scenario outlined in \cite[Section 3]{LSZ15NonabeHodexptwi}.
    \end{proof}
    
    \begin{theorem}\label{thm of catier correspondence}
    Let notations be as above. We have the equivalence of the following two categories,
        \begin{equation*}
        \begin{tikzcd}
            {\{\mathrm{MIC}_{lf,*}^{p-1}\}/(X,D)} & {\{\mathrm{HIG}^{p-1}_{lf,*}\}/(X',D')}.
            \arrow["{C_{\exp}}", shift left, harpoon, from=1-1, to=1-2]
            \arrow["{C_{\exp}^{-1}}", shift left, harpoon, from=1-2, to=1-1]
        \end{tikzcd}
        \end{equation*}
    \end{theorem}
    \begin{proof}
        Given that the exponential twists related to the Cartier and inverse Cartier are mutually inverse, then applying Theorem \ref{thm of Cartier Descent} finalizes the proof.
    \end{proof}

    \subsubsection{Functorial Properties of Parabolic Non-Abelian Hodge Correspondence}
    
The main result of this section focuses on the functorial properties of the non-abelian Hodge correspondence, as established in Theorem \ref{thm of catier correspondence}.
    \begin{theorem}\label{thm of par inverse commutes with pullback and push forward}
        Let $X$ and $Y$ be two smooth varieties over a perfect field $k$ of characteristic $p$, and $f: (X,D) \to (Y,B)$ be a separable branched covering with a $W_{2}(k)$-lifting $\tilde{f}: (\tilde{X},\tilde{D}) \to (\tilde{Y},\tilde{B})$, and we assume that all ramification indices are co-prime to $p$.
        %Let $B$ (resp. $\tilde{B}$) be a divisor of $Y$ (resp. $\tilde{Y}$), and let $D=(f^{*}B)_{\mathrm{red}}$ (resp. $\tilde{D}$) be the reduced divisor of $f^{*}B$ (resp. $\tilde{f}^{*}\tilde{B}$), 
        Then there exists a natural isomorphism
        \begin{equation}\label{pullback}
            \begin{tikzcd}
            {\{\mathrm{MIC}_{lf,*}^{p-1}\}/(X,D)} & {\{\mathrm{HIG}^{p-1}_{lf,*}\}/(X,D)} \\
            {\{\mathrm{MIC}_{lf,*}^{p-1}\}/(Y,B)} & {\{\mathrm{HIG}^{p-1}_{lf,*}\}/(Y,B),}
            \arrow["{C_{\exp}}", shift left, harpoon, from=1-1, to=1-2]
            \arrow["{C_{\exp}^{-1}}", shift left, harpoon, from=1-2, to=1-1]
            \arrow["{f^{*}}", from=2-1, to=1-1]
            \arrow["{C_{\exp}}", shift left, harpoon, from=2-1, to=2-2]
            \arrow["{f^{*}}"', from=2-2, to=1-2]
            \arrow["{C_{\exp}^{-1}}", shift left, harpoon, from=2-2, to=2-1]
            \end{tikzcd}
        \end{equation}
        and
        \begin{equation}\label{pushforward}
            \begin{tikzcd}
        	{\{\mathrm{MIC}_{lf,*}^{p-1}\}/(X,D)} & {\{\mathrm{HIG}^{p-1}_{lf,*}\}/(X,D)} \\
        	{\{\mathrm{MIC}_{lf,*}^{p-1}\}/(Y,B)} & {\{\mathrm{HIG}^{p-1}_{lf,*}\}/(Y,B).}
        	\arrow["{C_{\exp}}", shift left, harpoon, from=1-1, to=1-2]
        	\arrow["{C_{\exp}^{-1}}", shift left, harpoon, from=1-2, to=1-1]
        	\arrow["{f_{*}}"', from=1-1, to=2-1]
        	\arrow["{C_{\exp}}", shift left, harpoon, from=2-1, to=2-2]
        	\arrow["{f_{*}}", from=1-2, to=2-2]
        	\arrow["{C_{\exp}^{-1}}", shift left, harpoon, from=2-2, to=2-1]
        \end{tikzcd}
        \end{equation}
        Moreover, if $f$ is Galois with automorphism group $G$, we have
        \begin{equation*}
            \begin{tikzcd}
            	{\{\mathrm{MIC}_{lf,*}^{p-1}\}/(X,D)} & {\{\mathrm{HIG}^{p-1}_{lf,*}\}/(X,D)} \\
            	{\{\mathrm{MIC}_{lf,*}^{p-1}\}/(Y,B)} & {\{\mathrm{HIG}^{p-1}_{lf,*}\}/(Y,B).}
            	\arrow["{C_{\exp}}", shift left, from=1-1, to=1-2]
            	\arrow["{f_{*}^{G}}", shift left, from=1-1, to=2-1]
            	\arrow["{C^{-1}_{\exp}}", shift left, from=1-2, to=1-1]
            	\arrow["{f_{*}^{G}}", shift left, from=1-2, to=2-2]
            	\arrow["{f^{*}}", shift left, from=2-1, to=1-1]
            	\arrow["{C_{\exp}}", shift left, from=2-1, to=2-2]
            	\arrow["{f^{*}}", shift left, from=2-2, to=1-2]
            	\arrow["{C^{-1}_{\exp}}", shift left, from=2-2, to=2-1]
            \end{tikzcd}
        \end{equation*}
    \end{theorem}
\begin{proof}     
        It suffices to establish the functorial characteristics of the exponential twist, as the functorial properties related to Frobenius and Cartier descent were already demonstrated in Proposition \ref{Cartier equivalence commutes with pullback and pushforward}.
        We only have to show it for the exponential twist for inverse Cartier.
        Let $(E_{*},\theta)$ be a parabolic Higgs bundle on $(Y,B)$, with parabolic data
        \begin{equation*}
            (\{ Y_{i}\}_{i},\{\bm{e}\}_{i},\bm{\alpha}).
        \end{equation*}
        Then parabolic bases is $\bm{e}/t_{i}^{\alpha}$.
        The parabolic bases of the composition $C^{-1}_{(X,B)\subset (\tilde{X},\tilde{B})}\circ f^{*}(E_{*},\theta)$ is given by
        \begin{equation*}
        F^{*}f^{*}\bm{e}/s_{i}^{p\bm{n}\alpha},
        \end{equation*}
        where $\bm{n}$ denotes the ramification index.

        The parabolic bases of $f^{*}\circ C^{-1}_{(Y,D)\subset (\tilde{Y},\tilde{D})}(E_{*},\theta)$ is 
        \begin{equation*}
            f^{*}F^{*}\bm{e}/s_{i}^{p\bm{n}\alpha}.
        \end{equation*}
        The parabolic bases is naturally identified based on $F\circ f=f\circ F$.
        Since $\tilde{f}$ lifts $f$, we have $f^{*} \dif \tilde{F}_{U_{i}}^{*} = \dif \tilde{F}_{Y_{i}}^{*}$.
        Therefore, the parabolic connection $f^{*}(\nabla_{\mathrm{can}} + \frac{\dif F^{*}_{i}}{p} \circ F^{*}\theta_{\mathrm{par}})$ and $\nabla_{\mathrm{can}} + \frac{\dif F^{*}_{i}}{p} \circ f^{*}F^{*}\theta_{\mathrm{par}}$ can be identified, as well as the transition data.
        This proves the first isomorphism \eqref{pullback}.

        Let $(F_{*},\eta)$ be a parabolic Higgs bundle on $(X,D)$, with parabolic data
        \begin{equation*}
            (\{ U_{i}\}_{i},\{\bm{f}\}_{i},\bm{\beta}).
        \end{equation*}
        Then parabolic bases is $\bm{f}/s_{i}^{\beta}$.
        Therefore, the parabolic bases of $f_{*}(F,\eta)$ consists of the elements 
        \begin{equation*}
        \bm{f}s_{i}^{k}/t_{i}^{(\beta+k)/n},
        \end{equation*}
        where $n$ is the ramification index.
        The parabolic bases of $C^{-1}_{(X,B)\subset (\tilde{X},\tilde{B})}\circ f_{*}(E_{*},\theta)$ is 
        \begin{equation}\label{bases of pushforward of inverse Cartier}
            \frac{F^{*}\bm{f}s_{i}^{pk}}{t_{i}^{p\beta/n}\cdot t_{i}^{pk/n}}=\frac{F^{*}\bm{f}s_{i}^{\{pk/n\}n}}{t_{i}^{p\beta/n}\cdot t_{i}^{\{pk/n\}}},
        \end{equation}
        where $\{ \cdot \}$ is the decimal part function.
        The parabolic bases for $C^{-1}(F_{*},\eta)$ is $F^{*}f/s_{i}^{p\beta}$, and the parabolic bases for $f_{*}C^{-1}(F_{*},\eta)$ is
        \begin{equation*}
            \frac{F^{*}fs_{i}^{m}}{t^{p\beta/n}\cdot t^{m/n}},
        \end{equation*}
        which corresponds to \eqref{bases of pushforward of inverse Cartier} if we take $m=n\{ \frac{pk}{n}\}$.
        The transition and connection data can be naturally determined, therefore \eqref{pushforward} is proved.
    \end{proof}

    \begin{remark}
        Theorems \ref{thm of par inverse commutes with pullback and push forward} can be succinctly expressed by stating that the parabolic pullback and pushforward operations are compatible with the inverse Cartier transformation.
        Alfaya and Biswas proved an analogous result in the case $f$ is a finite morphism between two Riemann surfaces in \cite{AlBis23pullandpshofparbd}.
    \end{remark}

    Then we can prove that the inverse Cartier constructed above coincides with the one constructed by Krishnamoorthy-Sheng, the latter being denoted by $C^{-1}_{\para}$ in \cite{KrisSheng20perideRhamovercur}.

    \begin{corollary}\label{cor of two par inverse cartier coincide}
        $C^{-1}_{\exp}$ (resp. $C_{\exp}$) on $(X,D) \subset (\tilde{X},\tilde{D})$ is equivalent to $C^{-1}_{\para}$ (resp. $C_{\para}$). 
    \end{corollary}
    \begin{proof}
        We demonstrate the equivalence between $C_{\para}^{-1}$ and $C^{-1}_{\exp}$; the equivalence for the Cartier transformation follows similarly.
        According to the definition of $C^{-1}_{\para}$, it is sufficient to select a finite \'etale cover $f:(Y,B) \to (X,D)$ where the ramification indices are multiples of the respective denominators of the parabolic weights. 
        The existence of such a cover is guaranteed by the Kawamata-Viehweg lemma.
        The functor $C_{\para}^{-1}$ is defined through the following diagram:
        \begin{equation*}
	\begin{tikzcd}
	{\{ G\text{-parabolic Higgs bundle} \} / (Y',B') } & { \{ G\text{-parabolic flat connection} \} / (Y,B)} \\
	{\{ \text{parabolic Higgs bundle} \} /(X',D')} & { \{ \text{parabolic flat connection} \} /(X,D) }
	\arrow["{C^{-1}}", from=1-1, to=1-2]
	\arrow["f^{*}_{\para}"', from=1-1, to=2-1]
	\arrow["(f_{\para,*})^{G}", from=1-2, to=2-2]
	\arrow["{C^{-1}_{\para}}", from=2-1, to=2-2]
	\end{tikzcd}.
	\end{equation*}                
        By Theorem \ref{thm of Bis for vb with pb}, and \ref{thm of par inverse commutes with pullback and push forward}, we have
        \begin{equation*}
            C_{\para}^{-1}=(f_{\para,*})^{G}C^{-1}f_{\para}^{*}=(f_{\para,*})^{G}f_{\para}^{*}C^{-1}_{\exp}=C^{-1}_{\exp}.
        \end{equation*}
    \end{proof}    
    In order to prevent the introduction of extra notations, we shall henceforth eschew the use of $C_{\exp}^{-1}$, opting instead to denote it by $C_{\para}^{-1}$, or more succinctly as $C^{-1}$, provided that no ambiguity ensues.

    \begin{remark}\label{survey of positive char parabolic nonabelianHodge}
        This remark surveys recent advancements in extending non-abelian Hodge correspondence to parabolic and parahoric settings in positive characteristic. 
        As outlined in the introduction, the foundational parabolic non-abelian Hodge correspondence was established by Krishnamoorthy and Sheng in \cite{KrisSheng20perideRhamovercur}, generalizing the Ogus-Vologodsky correspondence to algebraic curves. 
        Subsequent work by Sheng, Sun, and Wang in \cite{ShengSunWangAnonabelianHodgecorforprincipalbdsinpositivechar} further extended this framework to principal bundles and parahoric bundles over higher-dimensional varieties. 
        Wakabayashi \cite{Wak24Frobeniuspullbackanddormantopers} independently recovered a generalized Ogus-Vologodsky correspondence for curves, incorporating higher-level \(\mathcal{D}\)-modules.  

        Building on Chen and Zhu’s foundational results in \cite{ChenZhunaHodgeforcurincharp} for curves in positive characteristic, recent studies have expanded logarithmic parabolic and parahoric non-abelian Hodge correspondences. 
        For instance, Shen \cite{Shen24TameramigeomeLaninposichar} established a correspondence between flat connections with regular singularities on a curve \(X\) and meromorphic Higgs bundles with first-order poles on its Frobenius twist \(X^{(1)}\). 
        Li and Sun \cite{LiSun24Tameparahoricnahcorreinpositivecharovercurve} proved a version for tame \(G\)-local systems over a curve \(C\) and logarithmic \(G\)-Higgs bundles over the Frobenius twist \(C'\). 
        Additionally, de Cataldo, André, and Siqing \cite{deCataAndreSiqing25LogarithmicNonAbelianHodgeTheoryforcurves} derived a non-parabolic logarithmic correspondence for curves over algebraically closed fields of characteristic \(p\), further broadening the scope of these results.  
    \end{remark}

\subsection{On Higgs-de Rham flow of rank \texorpdfstring{{\boldmath$2$}}{2}}	
This section is devoted to giving an algorithm to ascertain the maximal destabilizing subbundle within the foundational bundle of a parabolic flat connection.
    This parabolic flat connection is derived from the inverse Cartier transformation of the parabolic graded Higgs bundle of rank $2$ Higgs bundle, which generalizes the algorithm introduced by Sun-Yang-Zuo \cite[Appendix A]{SYZ22Projecrysrepoffmtgpandtwistedflow}.

We begin with a little general setting.
Let $X_{2}$ be a smooth variety with a normal crossing divisor $D_{2}$, let $(X_{1},D_{1})$ be its reduction modulo $p$. 
Let $ (E, \theta) $ (here and subsequently, $*$ is omitted for simplicity) be a graded parabolic Higgs bundle of length $ 2 $ of the form:
\begin{equation}\label{def of gaded Higgs bd}
\begin{split}
&E \cong L_1 \oplus L_2; \\
&\theta: L_{1} \rightarrow L_{2} \otimes \Omega_{X_{1}}^{1}(\log D_{1}).
\end{split}
\end{equation}
Denote $C_{(X_{1},D_{1})\subset (X_{2},D_{2})}^{-1}(E,\theta)$ by $(H,\nabla)$.
By the exactness of the inverse Cartier, we have the following short exact sequence
\begin{equation}\label{ext of inverse Cartier}
0 \to (F^{*}L_{2},\nabla_{\mathrm{can}}) \to (H,\nabla) \to (F^{*}L_{1},\nabla_{\mathrm{can}}) \to 0,
\end{equation}
where the canonical connection $\nabla_{\mathrm{can}}$ on $F^{*}L_{i}$ is defined by Proposition \ref{prop of Fro of par bd}. 
    
    By forgetting the connection, we have the short exact sequence of the bundle
    \begin{equation}\label{eq: ses of bd of flat connection}
        0 \to F^{*}L_{2} \to H \to F^{*}L_{1} \to 0.
    \end{equation}
    We denote the corresponding extension class in $\mathrm{Ext}^{1}(F^{*}L_{1},F^{*}L_{2})$ as $\xi$.
Recall the lifting $ (X_{1}, D_{1}) \hookrightarrow (X_{2}, D_{2}) $ gives rise to a cohomology class 
\begin{equation}\label{Deligne-Illusie class}
\kappa \in \mathrm{Ext}^{1}(F^{*}\Omega_{X_{1}}(\log D_{1}),\mathcal{O}_{X_{1}}) \cong H^{1}(X_{1}, F^{*}T_{X_{1}}(-\log D_{1})), 
\end{equation}
which is well-known as the Deligne-Illusie class.
    In fact, the Deligne-Illusie class can be expressed as the C\v{e}ch representation $[h_{ij}]$ as defined in \eqref{def of DI class}. 
    	
The next proposition reveals the relationship between the extension class, Deligne-Illusie class, and the Higgs field.      	
Before that, we recall the cup product of cohomology classes of coherent sheaves:
    For any two coherent sheaves $ \mathcal{F}$ and $\mathcal{G}$ over an ambient space $ X $, there exists a natural cup product in the cohomology groups:
\begin{equation}
H^{i}(X, \mathcal{F}) \times H^{j}(X, \mathcal{G}) \stackrel{\cup}{\rightarrow} H^{i + j}(X, \mathcal{F} \otimes \mathcal{G}).
\end{equation} 

\begin{theorem}\label{thm of ext class}
	The extension class of the inverse Cartier \eqref{ext of inverse Cartier} is the product of Frobenius of the Higgs field and the Deligne-Illusie class, \ie, we have
    \begin{equation} 
        \xi = F^{*}(\theta) \cup \kappa.
    \end{equation}
\end{theorem}
\begin{proof}
	It originates from the construction of the (parabolic) exponential twist.
        We pick an open cover $\{ U_{i} \}_{i\in I}$ of $X$ such that Frobenius lifting $\tilde{F}_{i}$ and parabolic bases exist on each $U_{i}$. 
        Let $\{\theta_{i}\}$ be the local expression of Higgs bundle, \ie, $\theta_{i} = \theta|_{U_{i}}: L_{1}|_{U_{i}} \to (L_{2} \otimes \Omega_{X_{1}}^{1}(\log D_{1}))|_{U_{i}}$, and let $\{ h_{ij}\}_{ij}$ be the C\v{e}ch representation of Deligne-Illusie class with respect to the cover $\{ U_{i} \}_{i\in I}$.
        By the definition of inverse Cartier, the C\v{e}ch representation of extension class $\xi=\{\xi_{ij}\}_{ij}$ is as follows: 
        \begin{equation}\label{ext class}
            \xi_{ij} = F^{*}(\theta_{i}) h_{ij}.
        \end{equation}
        Then by the definition of cup product, we have
        \begin{equation*}
            F^{*}\theta \cup \kappa = \{ F^{*}\theta_{i} \}_{i} \cup \{ h_{ij} \}_{ij} = \{ F^{*}\theta_{i} h_{ij}\}_{ij}.
        \end{equation*}
        combining \eqref{ext class} completes our proof. 
\end{proof}

Theorem \ref{thm of ext class} gives information about the underlying bundle $H$ of $C^{-1}(E,\theta)$. 
    We also offer an algorithm to determine the Harder-Narasimhan filtration of $H$.    
We restrict ourselves to the case where both $L_{1}$ and $L_{2}$ are line bundles.
The proposition \ref{thm of ext class} tells us the explicit formula of the extension class.
Next, we give a criterion for testing the maximal destabilizing subbundle of $H$.
    	
Let $ (H, \nabla) \cong C^{-1}(E, \theta) $, and let $ M $ be a sub-vector bundle of $H$.
    Then we apply $\mathrm{Hom}(M,*)$ to the \eqref{eq: ses of bd of flat connection}, we have the following exact sequence:
\begin{equation}\label{short exact sequence to algorithm}
    \begin{split}
0 \rightarrow \mathrm{Hom}(M,F^{*}L_{2}) \to \mathrm{Hom}(M,H) &\to \mathrm{Hom}(M,F^{*}L_{1}) \\
    &\stackrel{\delta}{\to} H^{1}(M,F^{*}L_{2}) \to \cdots. 
    \end{split}
\end{equation}

    \begin{proposition}\label{algorithm}
    $M$ is the maximal destabilizing sub-bundle of $H$ if and only if $\dim(\ker \delta)=1$.
    \end{proposition}
    \begin{proof}    	
Note that $M$ is the maximal destabilizing sub-bundle of $H$ if and only if 
    \begin{equation}\label{m d s}
    \dim(\mathrm{Hom}(M, H)) = 1.
    \end{equation}
    Moreover we have $\mathrm{Hom}(M,F^{*}L_{2})=0$, hence $\mathrm{Hom}(M,H) \cong \ker(\delta)$.
    The proof concludes by integrating \eqref{m d s}.
    \end{proof}
    
    If \(L_1\) and \(L_2\) are endowed with trivial parabolic structures, \(L_1 \otimes L_2 \cong \mathcal{O}_{X_1}\), the morphism \(\theta: L_1 \to L_2 \otimes \Omega_{X_1}^1(\log D_1)\) is an isomorphism, and the maximal destabilizing subbundle \(M\) is isomorphic to \(L_1\), then the triple \((H, \nabla, M)\) forms an \textit{oper}. 
    The explicit nature of the inverse Cartier transformation facilitates the accurate calculation of $\delta$ through the representation in C\v{e}ch cohomology.
    This proposition extends the methodology outlined in \cite[Appendix A]{SYZ22Projecrysrepoffmtgpandtwistedflow}. 
    An illustrative example detailing the computation process is provided in the appendix.
```

\section{Discussion}\label{Discussion}
The primary aim of this paper is to introduce the concept of parabolic bases and employ it to establish a parabolic non-abelian Hodge correspondence in positive characteristic.
This framework offers a comprehensive characterization of parabolic vector bundles and, via explicit computations, facilitates the derivation of key operations and results for vector bundles and $\lambda$-connections.
Given the local nature of the non-abelian Hodge correspondence in positive characteristic, the parabolic counterpart emerges naturally from this construction. We propose the following directions for future research:
\begin{enumerate}
\item Investigate whether parabolic bases can reconstruct Simpson’s non-abelian Hodge correspondence for quasi-projective complex curves \cite{Sim90Harbdonnoncomcur} and generalize it to higher-dimensional settings, as in Mochizuki’s work \cite{Mochi06KobaHitchcorrefortameharbdandapp,Mochi11Wildharbdandwildpuretwis}.
In particular, explore connections between parabolic bases and filtered local systems.
\item Extend the framework of parabolic bases to encompass torsion-free sheaves and coherent sheaves, broadening its applicability beyond vector bundles.
\item Leverage parabolic bases to study the geometry of root stacks—for instance, in contexts such as K-theory, operator theory, and related structures—given the correspondence between parabolic bundles and bundles over root stacks.
\item Having established the parabolic non-abelian Hodge correspondence in characteristic  $p$, a natural next step is to extend these results to the $p$-adic setting, thereby developing a comprehensive parabolic $p$-adic Hodge correspondence framework.
\end{enumerate}

%%%%%%%%%%%%%%%%%%%%%%%%%%%%%%%%%%%%%%%%%%%%%%%%%%%%%%%%%%%%%%%%%%%%%%

%%%%%%%%%%%%%%%%%%%%%%%%%%%%%%%%%%%%%%%%%%%%%%%%%%%%%%%
%%% Acknowledgements. ÖÂÐ»
%%%%%%%%%%%%%%%%%%%%%%%%%%%%%%%%%%%%%%%%%%%%%%%%%%%%%%%
\textbf{Acknowledgements:} The author would like to express gratitude to Jianping Wang, from whom the author learned the theory of parabolic bases, and to Zhaofeng Yu, as this paper originated from discussions with him. 
This work forms part of the author's thesis, and the author is deeply thankful to their advisor, Mao Sheng, for his guidance and support.  

The author also extends sincere thanks to the referee for their careful reading of the manuscript and for providing invaluable comments that greatly improved the paper.
 	
        \appendix	
        \section{On Li-Sheng conjecture}        
            As mentioned in the Introduction, the Gauss-Manin connection gives rise to a periodic Higgs-de Rham flow in a natural way. 
        In their work \cite{LiSheng22CharofBeanumber}, Li and Sheng introduced the concept of a periodic Higgs bundle in characteristic $0$. 
        This idea was further expanded to the parabolic context by Krishnammoorthy and Sheng in \cite{KrisSheng20perideRhamovercur}. 
        They put forward a conjecture suggesting that the periodic Higgs bundle is motivic. While the conjecture remains a long way from being proven, it provides significant insights into specific case analyses. 
        
        A primary example pertains to Beauville's elliptic fibration over projective line minus four points.
        We give a concrete exposition.
        Let $\lambda$ be a complex number, $(E,\theta)_{\lambda}$ be a logarithmic Higgs bundle on $(\mathbb{P}^{1}_{\mathbb{C}},0+1+\lambda+\infty)$ as follows:
        \begin{align*}
            &E \cong \mathcal{O}(1) \oplus \mathcal{O}(-1) ;\\
            &\theta : \mathcal{O}(1) \stackrel{\cong}{\to} \mathcal{O}(-1) \otimes \Omega_{\mathbb{P}^{1}}(0+1+\lambda + \infty),
        \end{align*}
        where $\mathcal{O}(d)$ denotes the unique bundle of degree $d$ on $\mathbb{P}^{1}_{\mathbb{C}}$.
        In \cite{LiSheng22CharofBeanumber}, the following question was posed: For which values of $\lambda$ can $(E, \theta)_{\lambda}$ be realized as the Kodaira-Spencer system of a semistable elliptic fibration $f: \mathcal{Y} \to \mathbb{P}^{1} - \{0, 1, \infty, \lambda\}$?
        Beauville proved there are only six families of such elliptic fibration, all of these turn out to be modular curves of genus zero.
        In Li and Sheng \cite{LiSheng22CharofBeanumber}, they aim to study these family in a Higgs-de Rham flow theoretic point of view.
        Their study is based on the following observation: If $(E,\theta)$ underlies a Kodaira-Spencer system of a semistable family, then it must be a periodic Higgs bundle for the reduction for almost all places of spreading out $(\mathbb{P}^{1}_{A},0+1+\lambda+\infty)$ of $(\mathbb{P}^{1}_{\mathbb{C}},0+1+\lambda+\infty)$, where $A$ is a finitely generated $\mathbb{Z}$-algebra.
        See \cite[Theorem 2.9]{LiSheng22CharofBeanumber} for a precise statement.
        Hence, the problem has been reduced to the study of the periodicity of $(E,\theta)$.
        It depends not only on the modulo of place $\mathfrak{p}$, but also on $A/\mathfrak{p}^{2} \cong W_{2}(A/\mathfrak{p})$, since the inverse Cartier depends on $W_{2}(k)$ lifting.
        We write the image of $\lambda$ in $A/\mathfrak{p}^{2}$ as the Witt vector $\lambda_{\mathfrak{p}^{2}} = (\lambda_{0},\lambda_{1})$.
        
        They proved the periodic condition if and only if $\lambda_{0}$ and $\lambda_{1}$ satisfy an algebraic relation \cite[Proposition 3.4]{LiSheng22CharofBeanumber}, which implies that $\lambda$ is an algebraic number.
        The appendix is devoted to simplifying this algebraic relation using Proposition \ref{m d s}.
        Let $(H,\nabla) \cong C^{-1}_{\lambda}(E,\theta)$.
	Assuming $ H \cong \mathcal{O}(a) \oplus \mathcal{O}(-a) $ with $ a \geq 0 $, since $ (H, \nabla) $ is gr-semistable, we get the inequality: $ a \leq -a + 2 $.
	Hence $ a = 1$ or $ a = 0 $.
 
	\textbf{Claim:} $ (E, \theta) $ is periodic if and only if $ a = 1 $.
	In fact, if $ a = 1 $ then $ \mathrm{Gr}(H, \nabla) \cong (\mathcal{O}(1) \oplus \mathcal{O}(-1), \theta') $, where $ \theta': \mathcal{O}(1) \rightarrow \mathcal{O}(-1) \otimes \Omega_{\Pl}^{1}(D) $.
	The stability condition forces $ \theta $ to be nonzero and thus $ (E,\theta) $ is one periodic.
	If $ a = 0 $,  then $ \mathrm{Gr}(H,\nabla) $ is not unique, but each of them are $ S $ equivalent to $ (\mathcal{O} \oplus \mathcal{O}, 0) $ then $ C^{-1}(\mathcal{O} \oplus \mathcal{O}, 0) \cong (\mathcal{O} \oplus \mathcal{O}, \dif) $ impossible to be periodic.
	Then we are in an amount to determine the case $a = 1$.
	We have the canonical short exact sequence 
	\begin{equation}
	0 \rightarrow (\mathcal{O}(-p), \nabla_{\mathrm{can}}) \rightarrow (H,\nabla) \rightarrow (\mathcal{O}(p), \nabla_{\mathrm{can}}) \rightarrow 0.
	\end{equation}            
	By forgetting the connection, we get the short exact sequence of the vector bundle.
	Taking the left exact functor $ \mathrm{Hom}(\mathcal{O}(1), *) $, we have the boundary map:
	\begin{equation}
	\mathrm{Hom}^{0}(\mathcal{O}(1),\mathcal{O}(p)) \stackrel{\Delta}{\rightarrow} \mathrm{R}^{1}\mathrm{Hom}(\mathcal{O}(1),\mathcal{O}(-p)) .
	\end{equation} 
	By Proposition \ref{algorithm}, $ H \cong \mathcal{O}(1) \oplus \mathcal{O}(-1) $ if and only if $ \Delta $ degenerates.
	In the sequel of this section, we aim to explicitly compute the map $ \Delta $.
	We first have to choose a bases for the vector space $H^{0}(\mathcal{O}(p-1))$ and $H^{1}(\mathcal{O}(p+1))$.
	Take an open cover $U_{1} = \mathbb{P}^{1}_{k}-\left\lbrace 0,\infty \right\rbrace $ and $U_{2}=\mathbb{P}^{1}_{k}-\left\lbrace1,\lambda \right\rbrace $ of $\mathbb{P}^{1}_{k}$.
	
	Let  
        \[
        R_{1} = \mathcal{O}(U_{1}) = k\left[x, \frac{1}{x}\right], \quad R_{2} = \mathcal{O}(U_{2}) = k\left[\frac{x-1}{x-\lambda}, \frac{x-\lambda}{x-1}\right],
        \]  
        and consider a representation of \(\mathcal{O}(k)\) such that \(\mathcal{O}(k)(U_{1}) = R_{1}\) and \(\mathcal{O}(k)(U_{2}) = \left(\frac{x-1}{x}\right)^{k} R_{2}\).  
        We then take the Čech representation of \(H^{i}(\mathcal{O}(k))\):  
        \[
        0 \rightarrow \prod_{i=1}^{2} \Gamma(U_{i}) \xrightarrow{d_{1}} \Gamma(U_{12}) \xrightarrow{d_{2}} 0,
        \]  
        which yields  
        \[
        H^{0}(\mathcal{O}(\ell)) = R_{1} \cap \left(\frac{x-1}{x}\right)^{\ell} R_{2}, \quad H^{1}(\mathcal{O}(-\ell)) = \frac{R_{12}}{R_{1} + \left(\frac{x}{x-1}\right)^{\ell} R_{2}}.
        \]  
        To make this explicit, we choose bases for the two vector spaces:  
        \[
        H^{0}(\mathcal{O}(\ell)) = k\left\{1, x, \dots, x^{\ell}\right\}, \quad H^{\ell}(\mathcal{O}(-k)) = k\left\{\left(\frac{x}{x-1}\right)^{\ell}, \dots, \left(\frac{x}{x-1}\right)^{\ell-1}\right\}.
        \]  
        Using these bases, we can now compute the corresponding matrix.              
        Let \(\tilde{F}_{i}\) be the lifting of Frobenius over \(U_{i}\) that preserves the divisor \(U_{i} \cap D\), i.e., \(\tilde{F}^{*}(\mathcal{O}(-V)) = \mathcal{O}(-p \, U_{i} \cap D)\). Explicitly, we have  
        \[
        \tilde{F}_{1}(x) = \frac{(x - \lambda_{\mathfrak{p}^{2}})^{p} - F(\lambda_{\mathfrak{p}^{2}})(x-1)^{p}}{(x - \lambda_{\mathfrak{p}^{2}})^{p} - (x-1)^{p}}, \quad \tilde{F}_{2}(x) = x^{p}.
        \]
	
	In \( U_{ij} = U_{i} \cap U_{j} \), the C\v{e}ch representation of the Deligne-Illusie class is given by  
        \begin{equation}\label{explicit Deligne-Illusie class}
        h_{12}(F^{*}\dif x) = \frac{\tilde{F}_{1}(x) - \tilde{F}_{2}(x)}{p},
        \end{equation}  
        and hence  
        \begin{equation*}
        h_{12}(F^{*}\dif x) = \frac{f_{\lambda_{0}}(x)(x^{p} - 1) - f_{1}(x)(x^{p} - \lambda_{0}^{p})}{\lambda_{0}^{p} - 1},
        \end{equation*}  
        where  
        \begin{equation*}
        f_{\lambda_{0}}(x) = \sum_{i=1}^{p-1} \frac{\lambda_{0}^{p-i}}{i} x^{i} + \lambda_{1}, \quad f_{1}(x) = \sum_{i=1}^{p-1} \frac{x^{i}}{i}.
        \end{equation*}  
        
        The matrix of \(\delta\), which is of size \(p \times p\), expressed in terms of the chosen bases, is  
        \begin{equation*}
        \Delta = \begin{Bmatrix}
        \lambda_{1} & \frac{-\lambda_{0}^{p} + \lambda_{0}^{p-1}}{p-1} & \frac{-\lambda_{0}^{p} + \lambda_{0}^{p-2}}{p-2} & \cdots & \frac{\lambda_{0}^{p} - \lambda_{0}}{1} \\
        \frac{\lambda_{0}^{p} - \lambda_{0}^{p+1}}{1} & \lambda_{1} & \frac{\lambda_{0}^{p} - \lambda_{0}^{p-1}}{p-1} & \cdots & \frac{\lambda_{0}^{p} - \lambda_{0}^{2}}{2} \\
        \vdots & & \vdots & & \vdots \\
        \frac{\lambda_{0}^{p} - \lambda_{0}^{2p-2}}{p-2} & \frac{\lambda_{0}^{p} - \lambda_{0}^{p+1}}{1} & \cdots & \lambda_{1} & \frac{-\lambda_{0}^{p} + \lambda_{0}^{p-1}}{p-1} \\
        \frac{\lambda_{0}^{p} - \lambda_{0}^{2p-1}}{p-1} & \frac{\lambda_{0}^{p} - \lambda_{0}^{2p-2}}{p-2} & \cdots & \frac{\lambda_{0}^{p} - \lambda_{0}^{p+1}}{1} & \lambda_{1}
        \end{Bmatrix}.
        \end{equation*}  
        From the explicit formula, it is evident that \(\det \Delta\) is a polynomial of degree \(p\) in \(\lambda_{1}\).  
        
        \begin{theorem}
        \((E, \theta)\) is periodic if and only if \(\det(\lambda_{0}, \lambda_{1}) = 0\).
        \end{theorem}  
        
        By a straightforward analysis of the degree of \(\det(\lambda_{0}, \lambda_{1})\), we obtain the following corollary:  
        
        \begin{corollary}
        For any \(\lambda \in \overline{F}_{P}\), there are exactly \(p\) liftings of \(\lambda\) such that \((E, \theta)\) is periodic.
        \end{corollary}  
        
        We now simplify the Li-Sheng conjecture (\cite[See Conjecture 1.5]{LiSheng22CharofBeanumber}) into the following statement:  
        
        \begin{conjecture}[Li-Sheng]
        \(\lambda\) is a Beauville number if and only if  
        \begin{equation*}
        \det M_{p}(\lambda_{0}, \lambda_{1}) = 0
        \end{equation*}  
        for almost all places \(\mathfrak{p}\).
        \end{conjecture}

\end{document}